\documentclass{aptpub}

\authornames{A. Malinowski, M. Schlather and Z. Zhang} 
\shorttitle{Intrinsically Weighted Means of Marked Point Processes}
\usepackage{mathrsfs} 
\usepackage{hyperref}   
\usepackage{enumerate}

\DeclareMathOperator{\erg}{erg}
\DeclareMathOperator{\Var}{Var}
\DeclareMathOperator{\Cov}{Cov}
\DeclareMathOperator{\g}{g}
\DeclareMathOperator{\weighted}{wght}
\DeclareMathOperator{\condit}{cond}
\DeclareMathOperator{\rel}{rel}
\DeclareMathOperator{\geo}{geo}
\DeclareMathOperator{\dis}{dist}

\newcommand{\dd}{d}
\newcommand{\bw}{\mathbf w}

\newcommand{\Eb}{\mathbb E}
\newcommand{\M}{\mathbb M}
\newcommand{\Nb}{\mathbb N}
\renewcommand{\P}{\mathbb P}

\newcommand{\R}{\mathbb R}
\newcommand{\T}{\mathbf T}
\newcommand{\Z}{\mathbb Z}
\newcommand{\1}{\mathbf 1}
\newcommand{\0}{\mathbf 0}
\newcommand{\A}{\mathcal A}
\newcommand{\BB}{\mathcal B}
\newcommand{\LL}{\mathcal L}
\newcommand{\MM}{\mathcal M}
\newcommand{\NN}{\mathcal N}
\newcommand{\cP}{\mathcal P}
\newcommand{\cPerg}{\mathcal P_{\erg}}

\begin{document}

\title{Intrinsically Weighted Means\\ of Marked Point Processes} 

\authorone[University of G\"ottingen]{Alexander~Malinowski}
\authortwo[University Mannheim]{Martin~Schlather}
\authorthree[University of Wisconsin at Madison]{Zhengjun~Zhang}
\addressone{Institute for Mathematical Stochastics,
   University of G\"ottingen, Goldschmidtstr. 7, 
   37077 G\"ottingen, Germany,
   +49 551 39 172134, 
   malinows@math.uni-goettingen.de}
\addresstwo{Institute for Mathematics, University Mannheim, 
  68131 Mannheim, Germany,
  schlather@math.uni-mannheim.de}
\addressthree{Department of Statistics, University of Wisconsin at Madison,
  1300 University Avenue, Madison, WI 53706,
 zjz@stat.wisc.edu}

\date{\today}

\begin{abstract} 
  \noindent 
  For a non-stationary or non-ergodic marked point process (MPP) on
  $\R^d$, the definition of averages becomes ambiguous as the process
  might have a different stochastic behavior in different realizations
  (non-ergodicity) or in different areas of the observation window
  (non-stationarity). We investigate different definitions for the
  moments, including a new hierarchical definition for non-ergodic MPPs,
  and embed them into a family of weighted mean marks.  We point out
  examples of application in which different weighted mean marks all
  have a sensible meaning.  Further, asymptotic properties of the 
  corresponding estimators are investigated as well as optimal weighting 
  procedures.
\end{abstract}


\section{Introduction} 

Marked point processes (MPPs) provide an adequate framework for
modeling irregularly scattered events in space or time in that they
incorporate the joint distribution of the observed values and the
point locations (e.g.,
\cite{Daley2008,Diggle2010,Karr1991,Moller2003,Myllymaki2007,Schlather2004}).
Due to the variety of possible forms of dependence between marks and
locations in an MPP framework, already the notion of the mean, which
is usually considered as being the simplest summary statistic, rises
tantalizing and challenging questions.

An introductory example for the type of MPP averages being considered
within this paper is the trading process in financial
markets. Transactions of assets are typically characterized by the two
quantities price and volume; a benchmark quantity that is of major
interest especially for institutional investors is the so-called
volume-weighted average price (VWAP) (e.g.,
\cite{Bialkowski2008,Madhavan2002}). The VWAP of $n$ transactions with
prices $p_i$ and traded volumes $v_i$, $i=1,\ldots, n$, is defined as
$p_{\text{VWAP}}= \sum (p_i v_i) / \sum v_i$.\\ We embed this example
in the following general MPP framework: We consider stationary MPPs on
$\R^d$ of the form
\begin{align*}
  \Phi = \left\{ (t_i, y_i, z_i) : i\in\Nb \right\}, 
\end{align*}
where $t_i\in\R^d$ is the point location, $y_i\in\R$ is the first mark
and $z_i\in[0, \infty)$ is a second mark of the $i$th point of
  $\Phi$. Let $\Phi_{\g}=\{t:(t,y,z)\in\Phi\}$ denote the ground
  process of point locations of $\Phi$ and let us denote the marks at
  a location $t\in\Phi_{\g}$ by $y(t)$ and $z(t)$. The non-negativity
  assumption on the $z$-component simplifies technical assumptions
  when employing this mark component as weights for averages of the
  first mark component $y(t)$ or $f(y(t))$ for some function
  $f:\R\to\R$.  In intuitive notation we write the corresponding
  weighted mean as
\begin{align}
  \mu_f^{(1)}=\Eb[z(t) f(y(t)) \,|\, t\in\Phi_{\g}],
  \label{wmm1}
\end{align}
where we assume that the $z$-component is normalized such that
$\Eb[z(t) \,|\, t\in\Phi_{\g}]=1$. Here, the conditioning on
``$t\in\Phi_{\g}$'' is understood in the sense of the Palm mark
distribution.  Since the weights $z(t)$ are provided by the MPP itself
and may depend on both the marks $y(t)$ and the point locations
$t\in\Phi_{\g}$, we refer to $\mu_f^{(1)}$ as \emph{intrinsically
  weighted mean mark} of $\Phi$.  The formal definition of
$\mu_f^{(1)}$ and related quantities will be given at the beginning of
Section \ref{sec:methods}.

When a system of randomly distributed objects is modeled by means of
MPPs, there can exist different sensible choices of intrinsic weights
$z(t)$ leading to different weighted mean marks that are relevant for
one and the same process, but for different statistical questions:

\begin{itemize}
  \item \emph{Average height of trees:} Consider $n$ forests of about
    equal size, each of which is sampled on an area with fixed size
    and shape.  Then the unweighted average of the height of all trees
    provides a measure of the entire timber stand, which is relevant
    for forest inventory applications. This amounts to $z(t)=1$ in
    \eqref{wmm1}.  Additionally, the average height of a typical
    \emph{forest} (as opposed to a typical tree) might be of interest,
    independently of how dense the trees occur in the different
    forests. Then, a nested definition of mean seems to be adequate
    where we first average within each forest and then between all
    forests. This is equivalent to using a weighted average over all
    trees with $z(t)$ being proportional to the inverse of the number
    of trees in the forest that location $t$ belongs to.
  \item \emph{Density of insects on plants, cf.\ \cite{Begon1990}:}
    Consider $n$ plants and a population of insects distributed over
    the plants.  Let $k_i$, $i=1, \ldots, n$, be the number of insects
    on the $i$th plant. In this set-up there are different
    well-established definitions of density referring to different
    ecological effects. The ordinary density of insects, also called
    \emph{resource-weighted density}, is $(k_1+\ldots+k_n)/n$ and
    quantifies the average availability of resources. In contrast, the
    \emph{organism-weighted density} is the density that an average
    insect experiences. Each individual on plant $i$ experiences a
    density of $k_i$ insects per plant, i.e., the organism-weighted
    density is $(k_1^2+\ldots +k_n^2) / (k_1+\ldots +k_n)$. In MPP
    notation, each insect is represented by a point, marked by the
    total number of insects on the plant on which the insect is
    located. Then the organism-weighted density corresponds to the
    ordinary mean mark ($z(t)=1$), whereas the resource-weighted
    density is the average of all plant-wise averages of the marks,
    i.e., $z(t)=(nk_i)^{-1} \sum_{i=1}^n k_i$ if $t$ belongs to
    plant~$i$.
  \item \emph{Sampling of continuous-space processes:} Measurements of
    continuous-space or continuous-time processes usually aim at
    estimating or predicting the underlying process and the mean of
    interest is therefore the spatial or temporal mean over the whole
    domain of the process. Since measurement locations are not
    necessarily independent of the underlying process, knowledge of
    the pattern of point locations might already provide information
    about the values of the process.  Such a situation is commonly
    referred to as \emph{biased} or \emph{preferential sampling} and
    different weighting approaches exist to correct for this form of
    biases (e.g., \cite{Isaaks1989}).  Although most statistical
    methods only use stationarity, ergodicity is often implicitly
    assumed. In case of non-ergodicity, which means that different
    realizations can have a different stochastic behavior, we are
    faced with an additional dimension of biasedness: Within each
    ergodic subclass, the pattern of point locations can be
    independent of the underlying process, while there might be a
    strong dependence between the pattern of measurement locations and
    the process itself if multiple realizations are considered.  For a
    simple example, consider a Gaussian random field with a random
    mean $m$ combined with a Poisson point process of measurement
    locations whose intensity of points is a function of~$m$.
\end{itemize}

While ergodicity of MPPs is necessary for a straightforward
interpretation of the mark distribution as the distribution of a
typical point and, at least implicitly, is required by many
applications for consistent estimation, in this paper, we investigate
the behavior of moment-based summary statistics in case of non-ergodic
MPPs and intend to point out problems of ambiguity in this
context. When the different forests and plants in the above examples
are perceived as a set of MPP realizations and exhibit different
ecological characteristics, non-ergodicity has to be
included. Examples for non-ergodic MPPs that evolve in time can easily
be found in the financial world: For subsequent days of asset trading,
the process of executed transactions can be considered as different
realizations of a possibly non-ergodic MPP.  To treat non-ergodic MPPs
adequately, we propose intrinsically weighted mean marks as a special
case of \eqref{wmm1} in which the weights are constant within each
ergodicity class but allow for compensating for differences between
the different ergodicity classes.  A direct application of the theory
developed within this paper is \cite{Malinowskietal2012b}, in which
interaction effects within high-frequency financial data are
investigated via MPP methods.

The remainder of this article is organized as follows: In Section
\ref{sec:methods} we recall and generalize moment-based
characteristics for MPPs which also form the central tool for the
analysis of interactions in MPPs. We study their behavior and
interpretation for non-ergodic processes and, following the idea of
the above examples, propose alternative definitions of moment-based
summary statistics in Section \ref{sec:nonergodic_def}. Different
estimators for the above characteristics and their asymptotic
properties are discussed in Section \ref{sec:estimators}; the paper
closes with a comparison of the point process set-up with estimation
of continuous-space processes, which typically occur within
geostatistical applications.  The appendix reviews basic results
from ergodic theory and contains some of the proofs of Section
\ref{sec:estimators}.


\section{MPP moment-measures and measurement of interaction effects}\label{sec:methods}
Throughout the paper $\Phi=\{(t_i, y_i, z_i): i\in\Nb\}$ is a
stationary and simple marked point process on $\R^d$ with marks
$(y(t_i), z(t_i)) = (y_i, z_i) \in \R\times[0, \infty)$, and
  $\Phi_{\g}=\{t: (t, y, z)\in\Phi\}$ is its ground process of point
  locations. In particular, the point configuration $\Phi_{\g}$ is
  locally finite.  For the general theory of point processes, the
  reader is referred to \cite{Daley2003,Daley2008,Stoyan1987}, for
  example.  Let us remark that the following definitions of MPP
  statistics can directly be generalized to MPPs on Polish spaces
  whose marks are also in a Polish space.

One of the most basic mark summary statistic is the weighted mean mark
$\mu_f^{(1)}$, which we introduced in \eqref{wmm1} as a conditional
mean, conditional on the event $\{t\in\Phi_{\g}\}$.  Since for fixed
$t\in\R^d$, this is a zero-probability event, the classical formal
definition is
\begin{align}
  \mu_f^{(1)} 
  = \frac{\Eb \sum_{(t, y, z)\in\Phi}z f(y) \1_B(t)}{
    \Eb \sum_{(t, y, z)\in\Phi} z\1_B(t)} 
  \label{mu_firstorder}
\end{align}
for any Borel set $B\subset\R^d$ with $|B|>0$. Here we implicitly
exclude the degenerate case $z(t)\equiv 0$. Due to the stationarity of
$\Phi$, this definition does not depend on the choice of $B$. 

\begin{prop}
Both definitions of $\mu_f^{(1)}$, \eqref{wmm1} and
\eqref{mu_firstorder}, coincide.
\end{prop}
\begin{proof}
The assertion follows from standard arguments of MPP theory  
\cite[chap.~13]{Daley2008}.
\end{proof}

The most relevant example of $f$ in practical application is $f(y) =
y^n$ for $n=1, 2, \ldots$ Then, if $z(t)=1$ for $t\in\Phi_{\g}$,
$\mu_f^{(1)}$ simply represents the $n$-th moment of the (Palm) mark
distribution.  Note that in case the MPP represents measurements of an
underlying continuous process, the mean mark can substantially differ
from the mean of the underlying process due to stochastic dependence
between the sampling locations and the process itself.

While the above statistic $\mu_f^{(1)}$ reflects (average) properties
of single points, second-order characteristics (in intuitive
notation $\Eb[f(y(t_1), y(t_2)) \,|\, t_1, t_2\in\Phi_{\g}, t_1\neq
  t_2]$) provide a framework to investigate dependency structures
within MPPs.  We use the superscripts $^{(1)}$ and $^{(2)}$ to
indicate whether first- or second-order measures are meant.

\begin{defn}
For any non-negative function $f$ on $\R\times\R$, we define 
a $\sigma$-finite measure on $\R^d\times\R^d$ by 
\begin{align}
  \alpha_f^{(2)}(C)=\Eb\sum_{(t_1, y_1, z_1), (t_2, y_2, z_2)\in\Phi}^{\neq}
  z_1 f(y_1, y_2)\1_C((t_1, t_2)), \quad C\in\BB(\R^d\times\R^d), 
  \label{alpha_secondorder2}
\end{align}
which we call \emph{weighted second moment measure}.
Here, ``$\neq$'' indicates that the sum runs over all pairs of
points with $(t_1, y_1) \neq (t_2, y_2)$.
\end{defn}
With the notation
\begin{align*}
C(B, I)&= \begin{cases}
  \big\{(t_1, t_2) : t_1\in B,\, t_2\in t_1+I \big\}, & d=1,\\
  \big\{(t_1, t_2) : t_1\in B,\, t_2\in t_1+\{x\in\R^d : \|x\|\in I\} 
\big\}, & d>1,
\end{cases}\\
C(t, I)&= C([\mathbf 0, t], I), \\
C(I) &= C([\mathbf 0, \mathbf 1], I),
\end{align*}
for $B\in\BB(\R^d)$, $t\in\R^d$, $I\in\BB(\R)$, 
\begin{align}
\alpha_f^{(2)}(C(I)),\quad I\in\BB(\R),\label{alpha_secondorder3}
\end{align}
defines a $\sigma$-finite measure on $\R$.  Well-known examples of
second-order mark characteristics for stationary and isotropic MPPs
are Cressie's mark variogram and covariance function
\cite{Cressie1993}, Stoyan's $k_{mm}$-function \cite{Stoyan1984a}, and
Isham's mark correlation function \cite{Isham1985}, which can all be
expressed in terms of \eqref{alpha_secondorder2} or
\eqref{alpha_secondorder3} with a constant $z$-component.
\cite{Schlather2001} provides a unifying notation for the above
characteristics and further introduces new functions, $E$ and $V$,
where $E(r)$ and $V(r)$ represent the mean and variance of a mark,
respectively, given that there exists a further point at distance
$r>0$. For the one-dimensional case, e.g., for temporal processes,
\cite{Malinowski2010} extend those characteristics to the
non-isotropic set-up, where a negative value of $r$ means that the
point that is conditioned on is in the past.  The above second-order
characteristics only involve the three functions $f(y_1, y_2) =
y_1y_2$, $f(y_1, y_2) = y_1$ and $f(y_1, y_2) = y_1^2$.
\begin{defn}[cf.\ \cite{Schlather2001}]
For a general non-negative function $f$ on $\R\times\R$, we define
\begin{align}
\mu_f^{(2)}(I)= \frac{\alpha_f^{(2)}(C(I))}{\alpha^{(2)}(C(I))}, 
\quad I\in\BB(\R),
\label{mu_f_discr}
\end{align}
if $\alpha^{(2)}(C(I)) >0$. Here, 
$\alpha^{(2)}$ is short notation for $\alpha_f^{(2)}$ with $f\equiv 1$. 
We call $\mu_f^{(2)}$ the \emph{(weighted) second-order mean mark}.
\end{defn}
In the following, we always assume that $I$ is chosen such that
$\alpha^{(2)}(C(I)) >0$.  Note that the distinction between $d=1$ and
$d>1$ in the definition of the set $C(B, I)$ allows to capture a
possibly anisotropic behavior of $\mu_f^{(2)}$ in the one-dimensional
case. In particular,
\begin{align*}
\alpha_f^{(2)}(C(I))=
\begin{cases} 
  \Eb_\Phi \sum^{\neq}_{(t_1, y_1, z_1), (t_2, y_2, z_2) \in\Phi,\ t_1\in [0,\, 1]} 
    z_1 f(y_1, y_2) \1_{t_2 - t_1 \in I}, & d=1\\
  \Eb_\Phi \sum^{\neq}_{(t_1, y_1, z_1), (t_2, y_2, z_2) \in\Phi,\ t_1\in [0,\, 1]} 
    z_1 f(y_1, y_2) \1_{\|t_2 - t_1\| \in I}, & d>1 .
\end{cases}
\end{align*}
For higher dimensions, it is also possible to assign different
directions of isotropy, but the technical burden increases
considerably as $\mu_f^{(2)}$ will not be a function of a scalar
argument anymore.  For further notational convenience, we assume that
the derivative of $\alpha_f^{(2)}$ w.r.t.\ the Lebesgue measure
exists, which is then referred to as \emph{product density} and
denoted by $\rho_f^{(2)}$.

Due to the stationarity of $\Phi$, we have $\rho_f^{(2)}(t_1, t_2) =
\rho_f^{(2)}(0, t_2-t_1)$ for almost all $(t_1, t_2)\in\R^{2d}$ and
hence $\alpha_f^{(2)}(C) = \int_{C} \rho_f^{(2)}(0, h_2-h_1) \dd
(h_1\times h_2)$, $C\in\BB(\R^d\times\R^d)$.  Let $\rho_f^{C,
  (2)}(r)$, $r\in\R$, denote the derivative of
$\alpha_f^{(2)}(C(\cdot))$ w.r.t.\ the \emph{one-dimensional} Lebesgue
measure.  Obviously, $\alpha_f^{(2)}(C(\cdot))$ is dominated by
$\alpha^{(2)}(C(\cdot))$, which ensures that the limit of
$\mu_f^{(2)}(I)$ for $|I|\to 0$ exists and can be expressed in terms
of Radon-Nikodym derivatives.  For $r\neq 0$ we define
\begin{align}
  \mu_f^{(2)}(r)= \left.
  \frac{
     \partial \alpha_f^{(2)}(C(\cdot)) 
     }{
     \partial \alpha^{(2)}(C(\cdot)) }\right|_{\cdot=r}
  = \frac{\rho_f^{C, (2)}(r)}{\rho_1^{C, (2)}(r)}.
  \label{mu_f_cont}
\end{align}
Note that for $d=1$, we have $\rho_f^{C, (2)}(r)=\rho_f^{(2)}(0, r)$.
With a slight abuse of notation, we refer to both definitions
(\ref{mu_f_discr}) and (\ref{mu_f_cont}) as $\mu_f^{(2)}$.  For $r\neq
0$ and $f$ only depending on its first argument, $\mu_f^{(2)}(r)$ can
be interpreted as the (weighted) expectation of a mark at location $t$
subject to the conditioning that $\Phi$ has a point at location $t$
and at location $t+r\mathbf e_1$, i.e., \hbox{$\mu_f^{(2)}(r) =
  \Eb[z(t) f(y(t)) \,|\, t, t+r\mathbf e_1\in\Phi_{g}]$}, where
$\mathbf e_1$ denotes the vector $(1, 0, \ldots, 0)^T\in\R^d$. For
$\mu_f^{(2)}(I)$, this interpretation becomes slightly ambiguous:
Considering an event at time $t$, there may be multiple other points
located within the set $t+I$ and in case that interactions of higher
order are present, these will be reflected by the second-order
statistic $\mu_f^{(2)}(I)$ as well. More precisely, by the definitions
in (\ref{mu_f_discr}) and (\ref{mu_f_cont}),
\begin{align}
  \mu_f^{(2)}(I) = \alpha^{(2)}(C(I))^{-1} \int_I \mu_f^{(2)}(r)
  \ \dd \alpha^{(2)}(C(r)) ,
  \label{mu_f_discr_smoothed}
\end{align}
i.e., $\mu_f^{(2)}(I)$ is a weighted average of conditional
expectations $\mu_f^{(2)}(r)$ with weights being proportional to 
the expected number of pairs of points with  distance  $\dd r$.

\begin{rem}\label{rem:generalizations}
\begin{enumerate}[(a)]
\item 
The extension to moment measures of higher order is 
straightforward and allows to condition on arbitrary point
constellations. In practice, however, mostly first- and second-order
statistics are considered.
\item 
The non-negativity condition on $f$ can be weakened by considering 
the restriction of $\mu_f^{(2)}(\cdot)$ to some bounded set $J\in\BB(\R)$.  
Then it is sufficient for $f$ that $\alpha_h^{(2)}(C(J)) <\infty$ is
satisfied for $h=f_+=\max\{f, 0\}$ \emph{or} for $h=f_-=-\min\{f, 0\}$.
\item 
Another generalization allows to include further conditioning on the
marks.  For $f_{\condit}$  a
non-negative function on $\R\times\R$ we consider 
\begin{align}
\mu_{f,\, f_{\condit}}^{(2)}(I)= 
\frac{\alpha_{f\cdot f_{\condit}}^{(2)}(C(I))}{\alpha_{f_{\condit}}^{(2)}(C(I))}
=\frac{\mu_{f\cdot f_{\condit}}^{(2)}(I)}{\mu_{f_{\condit}}^{(2)}(I)}.\label{mu_f_cond}
\end{align}
Choosing $f_{\condit}$ to be an indicator function
$f_{\condit}(y_1, y_2) = \1_A(y_1)\1_B(y_2)$ conditions the marks on
the events $A$ and $B$, respectively.
\end{enumerate}
\end{rem}

\begin{rem}
For $d>1$, $\mu_f^{(2)}$ is a function of the Euclidean distance
between two points, whereas for $d=1$, $\mu_f^{(2)}$ is a function of
the signed distance. In the latter case, $\mu_f^{(2)}(\cdot)$ is
in general not symmetric: Consider a temporal
process consisting of pairs of points $(t_1, t_2)$ with $t_1<t_2$
and with small intra- but large
inter-pair distances. Assume that the marks of different pairs are
stochastically independent and that for each pair of points, 
$f(y_1, y_2)>f(y_2, y_1)$ holds.
Then $\mu_f^{(2)}(r) > \mu_f^{(2)}(-r)$
holds for all $r>0$ that are small enough and that can occur as
intra-pair distances.
\end{rem}

For notational convenience, we will write $\mu_f^{(i)}$ to indicate that
a statement is valid for $\mu_f^{(1)}$ and $\mu_f^{(2)}$.

\section{New moment measures for non-ergodic MPPs}
\label{sec:nonergodic_def}
Ergodicity makes spatial averages over suitably increasing observation
windows of a single realization converge to the corresponding
expectation over the state space: $$|W|^{-1} \int_{W} X(T_x\Phi) \,\dd
x \stackrel{\text{a.s.}}{\longrightarrow} \Eb(X(\Phi)), \quad \text{
for } |W|\to\infty \text{ suitably},$$ for any integrable function $X$
on the space of all locally finite counting measures. Here, $T_x$
denotes the shift of the whole random point pattern $\Phi$ by
$x\in\R^d$.  In essence, ergodicity enables consistent estimation of MPP
moment measures by observing a single realization on a suitably
increasing domain.  In this section, though, we consider the opposite
situation, namely where $\Phi$ is a non-ergodic process.

\medskip

The following proposition directly relates to the fact that a
non-ergodic MPP can be seen as hierarchical model, which, in a first
step, draws an ergodic \emph{source of randomness} out of which the
final realization is drawn in a second step.
\begin{prop} \label{decomposition_mu}
Let $\Phi$ be a non-ergodic MPP with probability law $P$. By $\M_0$
and $\MM_0$ we denote the space of all locally finite counting
measures on $\R^d\times\R\times[0,\infty)$ and the usual
  $\sigma$-algebra, respectively.  (See Appendix \ref{ergodic_theory}
  for more details.)  Then
  \begin{align}
  \mu_f^{(1)} = \frac{
    \Eb_Q \left[ \mu_{f, \Phi | Q}^{(1)} \cdot \alpha_{\Phi|Q}^{(1)}(B) \right]
  }{ \alpha^{(1)}(B) } , \quad 
  \mu_f^{(2)}(\cdot) = \frac{
    \Eb_Q \left[ \mu_{f, \Phi | Q}^{(2)}(\cdot) 
    \alpha_{\Phi|Q}^{(2)}(C(\cdot)) \right]
  }{ \alpha^{(2)}(C(\cdot)) }, 
  \label{mu_f_summarize}
  \end{align}
where $Q\sim\lambda$ is a random variable with values in the space
$\cPerg$ of all ergodic MPP probability laws, distributed according to
some probability measure $\lambda$, such that
$P(M)=\int_{\cPerg}Q^*(M) \lambda(d Q^*)$, $M\in\MM_0$.  If
$\mu_f^{(2)}$ is evaluated for a fixed distance $r\in\R$,
$\alpha^{(2)}(C(r))$ has to be replaced by $\rho_{1}^{C, (2)}(r)$ 
in \eqref{mu_f_summarize}.
\end{prop}
\begin{proof}
The ergodic decomposition theorem (cf.\ Theorem \ref{decomp_theorem})
guarantees the existence and uniqueness of a decomposition
$P(\cdot)=\int_{\cPerg}Q^*(\cdot) \lambda(\dd Q^*)$ and a corresponding
mixing random variable $Q\sim\lambda$.  Conditioning $\Phi$ on $Q$, we
can decompose the moment measures $\alpha_f^{(i)}$ and obtain
\begin{align}
  \mu_f^{(2)}(r)
  &= \left. \frac{
    \partial \Eb_Q \alpha_{f, \Phi|Q}^{(2)}(C(\cdot))
  }{
    \partial \alpha^{(2)}(C(\cdot))
  }\right|_{\cdot=r}  
  = \frac{
    \partial \Eb_Q \alpha_{f, \Phi|Q}^{(2)}(C(\cdot)) 
    \big/ \partial \nu(\cdot) \Big|_{\cdot=r}
  }{
    \partial \alpha^{(2)}(C(\cdot))
    \big/ \partial \nu(\cdot) \Big|_{\cdot=r}
  } \notag\\[.3em]
  &= \frac{
     \Eb_Q \rho_{f, \Phi|Q}^{C, (2)}(0, r)
  }{
    \rho_1^{C, (2)}(0, r)
  } 
  = \frac{\Eb_Q\bigl[  \mu_{f, \Phi | Q}^{(2)}(r) \cdot  
      \rho_{1, \Phi|Q}^{C, (2)}(0, r)\bigr]
  }{ \rho_1^{C, (2)}(0, r) }, 
  \notag
\end{align}
where $\nu$ denotes the Lebesgue measure.
For $\mu_f^{(2)}(I)$ and $\mu_f^{(1)}$, the decomposition is analogous.
\end{proof}

\begin{ex}
The so-called log-Gaussian Cox process \cite{Moller1998}
is ergodic if and only if the
underlying stationary Gaussian random field $Z$ is
ergodic.
A sufficient condition for $Z$ being ergodic is that the covariance
function decays to zero.  Amongst others,
\cite{Diggle2010} and \cite{Myllymaki2007} 
use log-Gaussian Cox
processes, combined with an intensity-dependent marking, as parametric
models for preferential sampling applications.
\end{ex}

Proposition \ref{decomposition_mu} shows that in case of
non-ergodicity, $\mu_f^{(i)}$ is an average of its ergodic subclasses
counterparts, in which each class $Q^*$ is implicitly weighted by the
respective intensity $\alpha_{\Phi|Q=Q^*}^{(i)}$.  If all ergodic
subprocesses $[\Phi|Q=Q^*]$ have the same intensity measure, the
weights cancel out and we have \hbox{$\mu_f^{(i)}=\Eb_Q \mu_{f, \Phi |
    Q}^{(i)}$}.  Since in the general case, a single ergodicity class
with low probability may exhibit a large value of
$\alpha_{\Phi|Q=Q^*}^{(i)}$ and thus drive the value of $\mu_f^{(i)}$,
the demand for a new characteristic $\tilde \mu_f^{(i)}$ arises
naturally, that summarizes the properties of all ergodicity classes
irrespectively of how the processes of point locations differ between
the different ergodicity classes.  We meet these requirements by a
definition that excludes the implicit weighting proportional to the
$i$th order intensities:
\begin{defn}\label{def:mu_non-erg}
Let $\lambda$ and $Q$ be the ergodic decomposition mixture measure and
mixture variable, respectively, of $\Phi$, and let $\Eb_Q \big|\mu_{f,
  \Phi | Q}^{(i)}\big| < \infty$. Then we call
\begin{align}
  \tilde\mu_f^{(i)} &= \Eb_Q \mu_{f, \Phi | Q}^{(i)} 
  =\int_{\cPerg} \mu_{f, \Phi | Q=Q^*}^{(i)}\  \lambda(\dd Q^*).
  \label{mu_f_nonergodic}
\end{align}
the \emph{(equally-weighted) average $i$th-order mean mark} of $\Phi$.
\end{defn}
Relating to the introductory forest example, the classical definition
of the mean mark in \eqref{mu_firstorder} corresponds to the average
height of all trees, irrespectively of differences w.r.t.\ the tree
densities between the different forests, while the new definition in
\eqref{mu_f_nonergodic} refers to the average height of a typical
forest.

\begin{rem}
Comparing the new definition with (\ref{mu_f_summarize}) yields that
$\tilde\mu_f^{(i)}$ coincides with $\mu_f^{(i)}$ if
$\alpha_{\Phi|Q}^{(i)}$ is $\lambda$-a.s.\ constant.  This is
particularly the case if $\Phi$ is ergodic.
  \label{rem:const_alpha}
\end{rem}

\begin{lem}\label{lem:decomp_mu_tilde}
For any $I\in\BB(\R)$ we have
\begin{align*} 
\tilde\mu_f^{(2)}(I)
  &= \Eb_Q \left[\alpha^{(2)}_{\Phi | Q}(C(I))^{-1} \int_I \mu_{f, \Phi
  | Q}^{(2)}(r) \ \dd \alpha_{\Phi | Q}^{(2)}(C(r)) \right].
\end{align*}
If, for $\lambda$-almost all
  measures $Q^*$, $\mu_{f, \Phi | Q=Q^*}^{(2)}(r)$ is uniformly
  bounded by some positive constant $c(Q^*)$ and $\Eb_Q c(Q) <
  \infty$, for $I\in\BB(\R)$ and $r\in\R$, we have 
  \begin{align*}
  \lim_{I\to\{r\}} \tilde\mu_f^{(2)}(I) = \tilde\mu_f^{(2)}(r).
  \end{align*}
\end{lem}
\begin{proof}
The first assertion follows directly from applying the representation
\eqref{mu_f_discr_smoothed} to the ergodic subprocesses $[\Phi |
  Q=Q^*]$.  Since $\lim_{I\to\{r\}}\mu_f^{(2)}(I) = \mu_f^{(2)}(r)$ by
construction, the second assertion is merely an application of Lebesgue's
dominated convergence theorem.
\end{proof}

From Lemma \ref{lem:decomp_mu_tilde} we see that the nested conditional mean
$\tilde\mu_f^{(2)}(r)$ is a Radon-Nikodym derivative of
$\alpha_f^{(2)}(C(\cdot))$ w.r.t.\ $\alpha^{(2)}(C(\cdot))$ if and only if
the expectation of $\alpha^{(2)}_{\Phi | Q}(C(\cdot))\mu_{f, \Phi |
  Q}^{(2)}(\cdot)$ factorizes. This contrasts the
ordinary conditional mean $\mu_f^{(2)}(r)$, which is already defined
as a Radon-Nikodym derivative of $\alpha_f^{(2)}(C(\cdot))$
w.r.t.\ $\alpha^{(2)}(C(\cdot))$.

\medskip

The ergodic decomposition and an analog to Definition
\ref{def:mu_non-erg} can be applied to any expectation-based
functional of an MPP including the Palm mark distribution
itself. While the classical definition of the mean mark represents a
typical point, irrespectively of the different ergodicity classes, the
two-stage-expectation $\tilde\mu_f^{(i)}$ refers to the mean of a
typical realization.  We provide more details on the meaning of the
differences between $\mu_f^{(i)}$ and $\tilde \mu_f^{(i)}$ and between
different estimators in the next section.

\section{Estimation principles for the new MPP moment-measures}
\label{sec:estimators}

\subsection{The ergodic case}
For ergodic processes $\Phi$, the pointwise ergodic theorem for MPPs
(Proposition \ref{cor_averaging2} in the Appendix) yields that
\begin{align*}
& \Eb \left[ \sum^{\neq}_{(t_1, y_1, z_1), (t_2, y_2, z_2) \in\Phi} 
z_1 f(y_1, y_2) \1_{(t_1, t_2)\in C(I)}\right]\\
&\qquad\qquad
= \displaystyle \lim_{n\to\infty} \left[ n^{-d} 
  \sum^{\neq}_{(t_1, y_1, z_1), (t_2, y_2, z_2) \in\varphi} 
  z_1 f(y_1, y_2) \1_{(t_1, t_2)\in C(n\1, I)} \right]
\end{align*}
for almost all realizations $\varphi$ of $\Phi$, which builds the basis for 
the estimators being discussed in this section. 
For readability reasons, and since we will be only dealing with
second-order statistics from now on, we drop the superscript $^{(2)}$
in all the estimators of $\mu_f^{(2)}$.

Applying the standard estimator for MPP moment measures to a
realization of $\Phi$ observed on the set $[\0, \T]$, $\T\in(0, \infty)^d$, 
we obtain
\begin{align}
  \hat\mu_f(I, \Phi, \T)= 
  \frac{\hat\alpha_{f}(I, \Phi, \T)}{\hat\alpha_{1}(I, \Phi, \T)},
  \label{est_0}
\end{align}
where $\hat\alpha_{f}(I, \Phi, \T) = 
\sum^{\neq}_{(t_1, y_1, z_1), (t_2, y_2, z_2) \in\Phi} 
    z_1 f(y_1, y_2) \1_{(t_1, t_2)\in C(\T, I)}$.
\begin{lem}\label{lem_consistency} 
If $\Phi$ is ergodic, $\hat\mu_f(I, \Phi, \T)$ is consistent
for $\mu_f^{(2)}(I)$. Here, ``$\T\to\infty$'' is understood
componentwise. If $\Phi$ is non-ergodic, $\hat\mu_f(I, \Phi, \T)$ is 
consistent if and only if $\mu_{f, \Phi|Q=Q^*}^{(2)}(I)$ is constant w.r.t.\ 
$Q^*$.
\end{lem}
\begin{proof}
By Proposition \ref{cor_averaging2}, the tuple consisting of the
numerator and the denominator of \eqref{est_0}, each normalized by the
volume of $[0, \T]$, converges a.s.\ to the vector
$\big(\alpha_f^{(2)}(C(I)), \alpha^{(2)}(C(I))\big)$ if $\Phi$ is
ergodic. The first assertion thus follows from the continuous mapping
theorem. In the non-ergodic case, clearly only $\mu_{f,
  \Phi|Q=Q^*}^{(2)}(I)$ can be estimated consistently for $Q^*$ being
the respective ergodicity class. Though, if $\mu_{f,
  \Phi|Q=Q^*}^{(2)}(I)$ is constant w.r.t.\ $Q^*$ we have
$\mu_{f}^{(2)}(I)= \mu_{f, \Phi|Q=Q^*}^{(2)}(I)$ for any
$Q^*\in\cPerg$.
\end{proof}

\medskip

To establish asymptotic normality of $\hat\mu_f(I, \Phi, \T)$, we
introduce some idealized assumptions. In particular, we assume
stochastic independence between the point locations and the marks of
the MPP.  For simplicity, we restrict to the case where $f$ only
depends on its first argument and the MPP is a process on $\R$.
\begin{cond}[$m$-dependent Random Field Model]\label{RFmodel}
Let $\tilde\Phi$ be a stationary unmarked point process on $\R$, for
which neighboring points have some minimum distance $d_0 > 0$. Let
$\{Y(t) : t\in\R\}$ be an independent stationary process with finite
second moments and a covariance function $C$ that has finite range,
i.e., $C(h)=0$ for all $|h|>h_0$ for some $h_0>0$.  Then, with
$m=[d_0/h_0]$, we say that an MPP $\Phi$ is an \emph{$m$-dependent
  Random Field Model}, if $\Phi \stackrel{d}{=} \lbrace (t_i,
Y(t_i), 1)\,|\, t_i\in\tilde\Phi\rbrace.$
\end{cond} 
The following theorem transfers a central limit theorem (CLT) for
arrays of $m$-dependent random variables to the MPP context. It also
covers a thinning of the MPP in which the threshold increases with the
observation window.  The result allows to derive asymptotically exact
confidence intervals for the estimator of $\mu_{f}^{(2)}(I)$ and is
applied in \cite{Malinowskietal2012b} in the context of extreme value
analysis for MPPs.

\begin{thm}[CLT for $m$-dependent Random Field Models]
\label{CLT_MPP}
Let $\Phi$ be an ergodic MPP that satisfies Condition \ref{RFmodel}.  
For $f:\R\to[0,\infty)$ and $u\geq 0$, let $f_u, f_{\condit,
    u}:\R\to[0,\infty)$ be given by $f_u(y) = (f(y)-u)_+ = (f(y)-u)
    \1_{f(y)>u}$ and $f_{\condit, u}(y)=\1_{f(y)>u}$.  Let
\begin{align*}
\hat\alpha_{f_{u}}^*(I, \Phi, T) =
\sum^{\neq}_{(t_1, y_1), (t_2, y_2) \in\Phi} 
    \Bigl( f_u(y_1) - \mu_{f_u, f_{\condit, u}}^{(2)}(I) \Bigr)
    \cdot f_{\condit, u}(y_1)\cdot\1_{(t_1, t_2)\in C(\T, I)}
\end{align*} 
be a centered version of $\hat\alpha_{f_{u}}(I, \Phi, T)$, where 
$\mu_{f_u, f_{\condit, u}}^{(2)}(I)$ is defined as in \eqref{mu_f_cond}. 
Let
$(u_T)_{T\geq 0}$ be a family of non-negative, non-decreasing numbers
such that the following conditions are satisfied:
\begin{align*}
  u_\infty = \lim_{T\to\infty} u_T \in [0, \infty] \quad &\text{exists},\\
  \lim_{T\to\infty} \Eb\left[ f_{u_T}(Y(0))^i \big| 
    f(Y(0))>u_T\right] <\infty \quad & (i=1, \ldots, 4),\\ 
  \frac{T^{-1}\hat\alpha_{1}(I, \Phi, T) - \lambda}{
    \Eb_{\Phi}\hat\alpha_{f_{\condit, {u_T}}}(I, \Phi, 1)} 
  \to 0 \quad & \text{a.s. } (T\to\infty).
\end{align*}
Then, for $I\in\BB(\R)$ and $T\to\infty$, we have
\begin{align*}
  \frac{\hat\alpha_{f_{u_T}}^*(I, \Phi, T)}{
    \sqrt{\hat\alpha_{f_{\condit,u_T}}(I, \Phi, T)}} 
  &\cvgdist
  \NN(0, s_{u_\infty}),
\end{align*}
where 
\begin{align*}
  s_{u_\infty} &= \lim_{T\to\infty} \left\{(\lambda_{u_T} T)^{-1} 
  \Var\left[ \hat\alpha_{f_{u_T}}^{*}(I, \Phi, T) \right]\right\},\\
  \lambda_u &= \Eb_\Phi \left[\hat\alpha_{f_{\condit,u}}(I, \Phi, 1)\right], 
  \quad u\geq 0.
\end{align*}
\end{thm}
The proof is given in Appendix \ref{app_proof_CLT}.  Note that the
asymptotic variance $s_{u_\infty}$ can be given in a more explicit
form for suitable choices of $f$ and suitable distributional
assumptions on the underlying random field $Y$. A related CLT result
was provided by \cite{Heinreich1999} for random measures associated to
germ-grain models.

\subsection{The non-ergodic case}

If $\Phi$ is non-ergodic, consistent estimation of summary statistics
generally requires multiple realizations of the process.  Let $P$ and
$\lambda$ denote the probability law and the ergodic mixture measure
of $\Phi$, respectively.  Then, drawing iid realizations of $\Phi$
corresponds to drawing ergodicity classes according to the mixture
measure $\lambda$.  Though, a finite collection of realizations merely
approximates the mixing measure $\lambda$ and we can only expect
consistency if both $n$ and $\T$ tend to infinity simultaneously. To
see why $n \to \infty$ is not sufficient, consider an MPP with
infinitely many ergodicity classes $Q_1, Q_2, \ldots$ and with
$\Eb_{\Phi|Q=Q_i} \Phi([0, 1])=2^{-i}$.  Then, for fixed $\T$, the
probability of observing at least one point in a realization that
belongs to class $i$ tends to zero as $i\to\infty$. Hence, the classes
$Q_i$, for $i$ large, are only captured by the estimator
if $\T$ also tends to infinity.

Considering  iid realizations $\Phi_1, \ldots, \Phi_n$ of $\Phi$, different
possibilities arise of how to put together the respective estimators.
Let $\bw=(w_1, \ldots, w_n)$ denote a vector of
weight functions $w_i:\M_0\times [0, \infty)^d \to [0, \infty)$. We
assume that for $\lambda$-almost all ergodic MPP laws $Q^*$ there
exist constants $w_i^*(Q^*)\geq 0$ with $w^*(Q^*) = \sum_{i=1}^n
w_i^*(Q^*) > 0$ to which the weights converge stochastically 
within the respective ergodicity class, i.e., 
\begin{align}
  P_{\Phi|Q=Q^*}\left( |w_i(\Phi, \T) - w_i^*(Q^*)| > \varepsilon \right) 
  \longrightarrow 0  \quad (\T\to\infty) 
  \label{weights_convergence}
\end{align}
for all $\varepsilon >0$. Then we
consider estimators of the form
\begin{align}
  \hat\mu_f^{n, \weighted}(I, \bw)  
  &= \hat\mu_f^{n, \weighted}(I, \bw, (\Phi_1,\ldots,\Phi_n), \T) 
  \notag\\ 
  &= \left(\sum w_i(\Phi_i, \T)\right)^{-1} 
  \sum_{i=1}^n w_i(\Phi_i, \T)  \hat\mu_f(I, \Phi_i, \T),
  \label{est_2}
\end{align}
Note that the functions $w_i$ might also depend on $I$.
With  $w_1=\ldots=w_n = n^{-1}$, we obtain as a special case
\begin{align}
  \hat\mu_f^n(I) = \hat\mu_f^n(I, (\Phi_1,\ldots,\Phi_n), \T) 
  = n^{-1} \sum_{i=1}^n \hat\mu_f(I, \Phi_i, \T).
  \label{est_1}
\end{align}

In order to estimate $\mu_f^{(2)}(I)$ consistently, 
according to the decomposition in (\ref{mu_f_summarize}),  
the weights have essentially to be chosen as
\begin{align}
w_i(\Phi_i, \T)=\hat\alpha^{(2)}(C(\T, I), \Phi_i) / v_\T
= \sum^{\neq}_{t_1, t_2\in\Phi_{i, \g}} \1_{(t_1, t_2)\in C(\T, I)} / v_\T,
\label{weights_alpha}
\end{align}
where $v_\T$ is the volume of the cube $[\0, \T]$. By
Proposition \ref{cor_averaging2}, $\hat\alpha^{(2)}(C(\T, I), \Phi_i)
/ v_\T$ converges to $\alpha_{\Phi|Q=Q_i}^{(2)}(C(I))$ a.s.\ as
$\T\to\infty$, where $Q_i$ is the realized ergodicity class of $\Phi_i$.  
With $\bw$ being the vector of weights from
\eqref{weights_alpha}, we define
\begin{align}
  \hat\mu_f^{\alpha}(I, (\Phi_1,\ldots,\Phi_n), \T) = 
   \hat\mu_f^{n, \weighted}(I, \bw, (\Phi_1,\ldots,\Phi_n), \T),
  \label{est_concat}
\end{align}
which, in a sense,
represents the family of \emph{all} pairs of points with a distance 
contained in $I$ from all realizations.
This choice of weights satisfies the above stochastic
convergence condition (\ref{weights_convergence}) and is sufficient but not
necessary for consistency. The following theorem gives a weaker 
set of conditions that 
is still sufficient for consistency
\begin{thm}\label{prop_consistency_weighted_est}
Let $\Phi_i$, $i\in\Nb$, be iid copies of a possibly non-ergodic MPP
$\Phi$ and let $Q_{j_i}$ denote the respective ergodicity classes.
For weight functions $\tilde w_i:\M_0\times [0, \infty)^d \to [0,
    \infty)$ and iid random factors $W_i$ with 
$\Eb|W_i|<\infty$, $i\in\Nb$, let $w_i(\Phi_i,
    \T) = W_i \cdot \tilde w_i(\Phi_i, \T)$ and $\bw=(w_1(\Phi_1, \T),
    \ldots, w_n(\Phi_n, \T))$.  Then, $\hat\mu_f^{n, \weighted}(I,
    \bw)$ is consistent for $\mu_f^{(2)}(I)$ if the following
    conditions hold:
\begin{align}
  W_i &> 0\qquad\text{a.s.},\label{wi_1}\\
  \Var \tilde w_i(\Phi_i,\T) &\leq c_1 \quad\, \text{ for some } c_1>0,\label{wi_2}\\
n^{-1} \Eb \sum_{i=1}^n \tilde w_i \geq c_2&>0 \qquad \forall n\geq n_0 \text{ for some } n_0\in\Nb,\label{wi_3}\\
  \Eb\left[ W_i \tilde w_i(\Phi_i,\T)\right]
  &= \Eb\left[ W_i\right] \cdot \Eb\left[\tilde w_i(\Phi_i,\T)\right] \label{wi_indep1}\\
  \Eb\left[ W_i \cdot \hat\alpha^{(2)}(C(\T, I),\Phi_i)\mu_{f, \Phi|Q=Q_{j_i}}^{(2)}(I)\right]
  &= \Eb\left[ W_i\right] \cdot \Eb\left[\hat\alpha^{(2)}(C(\T, I),\Phi_i)\mu_{f, \Phi|Q=Q_{j_i}}^{(2)}(I)\right] \label{wi_indep3}\\[-3.5em]\notag
\end{align}
\begin{align}
  \P\Biggl\{ \max_{i=1}^n\left|
  \frac{\tilde w_i(\Phi_i,\T) \sum_{j=1}^n \hat\alpha^{(2)}(C(\T, I),\Phi_j)}{
    \hat\alpha^{(2)}(C(\T, I), \Phi_i)\sum_{j=1}^n \tilde w_j(\Phi_j,\T)}
  \right|  &> c_3 \Biggr\} \to 0 \quad (n, \T \to\infty) 
  \label{wi_1aa}
\end{align}
\end{thm}
\begin{proof}
We consider 
\begin{align}
  &\left| \frac{\sum_{i=1}^n w_i(\Phi_i, \T)  \hat\mu_f(I, \Phi_i, \T)}{
    \sum_{i=1}^n w_i(\Phi_i, \T)} - \mu_f^{(2)}(I) \right|\notag\\
  &\qquad\leq  \left| \frac{\sum_{i=1}^n w_i(\Phi_i, \T)  
    \big[\hat\mu_f(I, \Phi_i, \T) - \mu_{f, \Phi|Q=Q_{j_i}}^{(2)}(I)\big]}{
    \sum_{i=1}^n w_i(\Phi_i, \T)} \right|\label{tri_sum_1}\\ 
 &\hspace*{12em}
  +  \left| \frac{\sum_{i=1}^n W_i \tilde w_i(\Phi_i, \T)  
    \mu_{f, \Phi|Q=Q_{j_i}}^{(2)}(I)}{
    \sum_{i=1}^n W_i \tilde w_i(\Phi_i, \T)} - \mu_f^{(2)}(I)\right|
  \label{tri_sum_2} 
\end{align}
By Lemma \ref{lem_consistency}, $\hat\mu_f(I, \Phi_i, \T)$ is
consistent (for $\T\to\infty$) within the respective ergodicity
class. Thus, \eqref{tri_sum_1} converges to 0 in probability if
$\T\to\infty$. Using the short notation $\alpha_i =
\hat\alpha^{(2)}(C(\T, I), \Phi_i)$ and $\tilde w_i = \tilde
w_i(\Phi_i, \T)$, we have
\begin{align*}
  \eqref{tri_sum_2} &=   
   \left| \sum_{i=1}^n \frac{ W_i \alpha_i
    \big[\mu_{f, \Phi|Q=Q_{j_i}}^{(2)}(I)- \mu_f^{(2)}(I)\big]}{
    \sum_{j=1}^n W_j \alpha_j}
  \cdot\frac{\tilde w_i \sum_{j=1}^n W_j\alpha_j}{
    \alpha_i\sum_{j=1}^n W_j\tilde w_j}
  \right|\\
  &\leq \max_{i=1}^n\left\{ \left|\frac{\tilde w_i \sum_{j=1}^n \alpha_j}{
    \alpha_i\sum_{j=1}^n \tilde w_j}\right|\right\}
  \cdot\left|\frac{\sum_{j=1}^n W_j\alpha_j}{\sum_{j=1}^n \alpha_j}\right|
  \cdot \left|\frac{\sum_{j=1}^n \tilde w_j}{\sum_{j=1}^n W_j\tilde w_j}
  \right| \\
  & \hspace*{15em}
  \cdot \left|  \frac{\sum_{i=1}^n W_i \alpha_i
    \big[\mu_{f, \Phi|Q=Q_{j_i}}^{(2)}(I)- \mu_f^{(2)}(I)\big]}{
    \sum_{i=1}^n W_i \alpha_i}
  \right|
\end{align*}
Since by assumption, $(n^{-1} \Eb \sum_{i=1}^n \tilde
w_i)_{n\in\Nb}$
is eventually bounded away from 0 and the variance of the $\tilde w_i$
is uniformly bounded, the law of large numbers yields that
$\sum_{j=1}^n \tilde w_j / \Eb\sum_{j=1}^n\tilde w_j$ and
$\sum_{j=1}^n \tilde W_j w_j / \Eb\sum_{j=1}^n W_j\tilde w_j$ converge
to 1 in probability.  Additionally using that $\Eb[W_j \tilde w_j] =
\Eb W_j \Eb \tilde w_j$, for $n\to\infty$, we get the convergence
\begin{align*}
 \frac{\sum_{j=1}^n \tilde w_j}{\sum_{j=1}^n W_j\tilde w_j}
&=\frac{\sum_{j=1}^n \tilde w_j / \Eb\sum_{j=1}^n\tilde w_j}{
   \sum_{j=1}^n W_j\tilde w_j / \Eb\sum_{j=1}^n W_j\tilde w_j}
 \cdot \frac{\Eb\sum_{j=1}^n\tilde w_j}{\Eb\sum_{j=1}^n W_j\tilde w_j} 
\stackrel{p}{\longrightarrow} \frac{1}{\Eb W_1}
\end{align*}
as $n\to\infty$. Similarly, for $n\to\infty$ and $n,\T\to\infty$, we have
\begin{align*}
  \frac{\sum_{j=1}^n W_j\alpha_j}{ \sum_{j=1}^n \alpha_j}
  &\stackrel{p}{\longrightarrow} \frac{\Eb [W_1\alpha_1]}{\Eb\alpha_1},\\  
  \frac{\sum_{i=1}^n W_i\alpha_i \mu_{f, \Phi|Q=Q_{j_i}}^{(2)}(I)}{
    \sum_{i=1}^n W_i\alpha_i}
  &\stackrel{p}{\longrightarrow} 
  \frac{\Eb \left[ \alpha^{(2)}_{\Phi | Q=Q_{j_i}}(C(I)) \cdot \mu_{f, \Phi|Q=Q_{j_i}}^{(2)}(I)\right]}{\Eb \left[ \alpha^{(2)}_{\Phi | Q=Q_{j_i}}(C(I))\right]} = \mu_f^{(2)}(I),
\end{align*}
respectively.
Together with \eqref{wi_1aa} we obtain that \eqref{tri_sum_2} converges to 0 in probability, which completes the proof.
\end{proof}
Note that if $\tilde w_i =\tilde w$ for all $i\in\Nb$ for some weight
function $\tilde w$ with $\Eb|\tilde w(\Phi, \T)|<\infty$, 
the $\tilde w_i(\Phi_i,\T)$ are iid and
conditions \eqref{wi_2}, \eqref{wi_3} and \eqref{wi_indep1} become
obsolete.

\medskip

 Now we turn to the estimation of $\tilde\mu_f^{(2)}(I)$.  By
 construction (cf.\ Definition \ref{def:mu_non-erg}), $\hat\mu_f^n(I)$
 consistently estimates $\tilde \mu_f^{(2)}(I)$; in contrast to
 $\hat\mu_f^{\alpha}(I)$, it reflects a random pair of points with
 distance $I$ within a randomly chosen ergodicity class.  Again, also
 other choices of weights are feasible for consistent estimation of
 $\tilde \mu_f^{(2)}(I)$, apart from the choice $w_i(\Phi_i, \T) = 1$.
 By replacing $\hat\alpha^{(2)}(C(\T,I), \Phi_i)$ by the constant~1 in Theorem \ref{prop_consistency_weighted_est}, we get the following corollary.
\begin{cor}
Under the assumptions of Theorem
\ref{prop_consistency_weighted_est} with $\hat\alpha^{(2)}(C(\T,I),
\Phi_i)$ being replaced by the constant~1, $\hat\mu_f^{n,
  \weighted}(I, \bw)$ is a consistent estimator for $\tilde
\mu_f^{(2)}(I)$.
\end{cor}

\begin{rem}
If $\Phi$ is ergodic, $\hat\mu_f^{n, \weighted}(I, \bw)$ is consistent
for $\mu_f^{(2)}(I)$ (as $\T\to\infty$) for any choice of weights
$\bw$ that satisfies \eqref{weights_convergence}.  Note that in this
case, consistency is independent of $n$, which can be fixed to any
finite value.
\end{rem}
\begin{proof}
If $\Phi$ is ergodic, the mixing measure $\lambda$ is the one-point
distribution $\delta_P$ and condition (\ref{weights_convergence})
simply means stochastic convergence of the weights w.r.t.\ $P$. The
assertion directly follows from the continuous mapping theorem.
\end{proof}

\subsection{Variance minimization}

In what follows, we seek for an optimal consistent estimator for
$\tilde\mu_f^{(2)}(I)$ in the sense of minimal variance. We introduce
some additional assumptions on the mark-location dependence for
analytical tractability. For simplicity, we set $\tilde
w_i(\Phi_i,\T)=1$, i.e., we consider $w_i(\Phi_i,\T)=W_i$. Let
$\A_n^*$ denote the $\sigma$-algebra generated by the unmarked ground
processes $\Phi_{1, \g}, \ldots, \Phi_{n, \g}$, i.e., $\A_n^* =
\sigma(\{ \{\omega:\Phi_{i, \g}(\omega)(B)=k\} : k\in\Nb, B\in\BB,
i=1, \ldots, n \})$.  We assume that $\Eb[\hat\mu_f(I,\Phi_i,\T)\,|\,
  \A_n^*]$ is a.s.\ constant.  We further assume that $\A_n^*$ is
maximal w.r.t.\ this property and that $\Var \left[\left. \hat\mu_f(I,
  \Phi, \T) \right| \A_n^* \right]$ is independent of the random
ergodicity class~$Q$.
\begin{prop}\label{prop:var_min}
With the above notation and assumptions, 
the variance minimizing weights for
$\hat\mu_f^{n, \weighted}(I, \bw, (\Phi_1,\ldots,\Phi_n), \T)$ that satisfy 
(\ref{wi_1})--(\ref{wi_1aa}) with
$\hat\alpha^{(2)}(C(\T, I), \Phi_i)$ being replaced by~1 are given by 
\begin{align*}
  w_i(\Phi_i, \T) =W_i 
  =\Var \left[\left. \hat\mu_f(I, \Phi_i, \T) \right| \A_n^* \right]^{-1}.
\end{align*}
\end{prop}
Note that an analog variance minimizing procedure via random factors
$W_i$ could also be included into the estimator $\hat\mu_f^{\alpha}$ of
$\mu_f^{(2)}(I)$.
\begin{proof}[Proof of Proposition \ref{prop:var_min}]
For general $\A_n^*$-measurable weights
$w_i(\Phi_i, \T)$, $i=1,\ldots,n$, we have
\allowdisplaybreaks
\begin{align}
  \Var &
  \left[
    \hat\mu_f^{n, \weighted}(I, \bw, (\Phi_1,\ldots,\Phi_n), \T)
  \right] \notag\\
  &= \Eb \left[ 
    \frac{1}{\left(\sum w_i(\Phi_i, \T)\right)^{2}} 
    \sum_{i=1}^n w_i(\Phi_i, \T)^2
    \Var \left[\left.  \hat\mu_f(I, \Phi_i, \T) 
      \right| \A_n^* \right]
    \right]\notag\\
  &+ \Var \left[
    \frac{1}{\sum w_i(\Phi_i, \T)} 
    \sum_{i=1}^n w_i(\Phi_i, \T) 
    \Eb \left[\left.  \hat\mu_f(I, \Phi_i, \T) 
      \right| \A_n^* \right]
    \right]\notag\\
  &= \Eb \left[ 
    \sum_{i=1}^n w_i^{\rel}(\Phi_i, \T)^2
    \Var \left[\left.  \hat\mu_f(I, \Phi_i, \T) 
      \right| \A_n^* \right]
    \right]  \quad + \quad 0\label{var_mu_f}
 \end{align}
with $w_i^{\rel}(\Phi_i, \T) = w_i(\Phi_i, \T) / \sum_{i=1}^n w_i(\Phi_i, \T)$.
Since any weighted average $\sum v_i^2 x_i$ with $x_i>0$ and $\sum
v_i=1$ is minimized by $v_i= x_i^{-1}/\sum x_i^{-1}$ (Lagrange method),
the unconditional variance  (\ref{var_mu_f}) is minimized by choosing
\begin{align*}
  w_i(\Phi_i, \T) =W_i 
  =\Var \left[\left. \hat\mu_f(I, \Phi_i, \T) \right| \A_n^* \right]^{-1}.
\end{align*}
The $W_i$ are $\A_n^*$-measurable by definition of the conditional
variance and satisfy (\ref{wi_1})--(\ref{wi_1aa}) with
$\hat\alpha^{(2)}(C(\T, I), \Phi_i)$ being replaced by~1.
Maximality of
$\A_n^*$ ensures optimality of the weights.
\end{proof}

\medskip

If there exist interaction effects in the
MPP that are of higher than second order, 
the assumption on $\Eb[\hat\mu_f(I,\Phi_i,\T)\,|\, \A_n^*]$ might
not be satisfied anymore and weighting according to the above
conditional variances should be handled with care. 
Clusters of point locations which tend to increase the conditional
variance of $\hat\mu_f$ given the ground process, can additionally
influence the mean of other marks in excess of the bivariate
interaction measured by $\mu_f^{(2)}(I)$. Then, a bias will be
introduced by using the above random weights.  More generally, the
more is known about the relation between $\hat\mu_f(I, \Phi, \T)$ and
the ground process $\Phi_{\g}$, the more can be gained from using
different (random) weights while preserving consistency 
of the estimator. Without any assumption, only
deterministic or independent weights are feasible and then
$w_i(\Phi_i, \T)=1$ is naturally the best choice, i.e., the use of
$\hat\mu_f^n(I)$.

We consider two simple examples of optimal weighting in the
following. Here we assume that the $z$-components of the marks are 1
for all points.  Recall that $\hat\mu_f(I, \Phi, \T)=
\hat\alpha_{f}(I, \Phi, \T) \big/ \hat\alpha_{1}(I, \Phi, \T)$, that
the denominator is $\A_n^*$-measurable, and that $\hat\alpha_{f}(I,
\Phi_i, \T)$ is a sum consisting of $\hat\alpha_{1}(I, \Phi_i, \T)$
random summands.

\begin{rem}\label{ex:iid_case}
In general, the summands of $\hat\alpha_{f}(I, \Phi_i, \T)$ are not iid. 
However, if conditionally on $\A_n^*$, the summands   
\emph{were} iid with variance $v$, the conditional variance
 $\Var\left[\hat\mu_f(I, \Phi, \T) \,|\, \A_n^*\right]$ would be 
$v/\hat\alpha_{1}(I, \Phi_i, \T)$.
\end{rem}
In the following scenarios, we assume $f$ to depend on its first
argument, only. The proofs are given in Appendix \ref{app_proof_ex}.
\begin{ex}\label{prop_toyexample1} 
Let $\Phi$ have marks that are stochastically independent of the
process of point locations and let these point locations be fully
regularly spaced in every realization. Let $v_\T$ and $N=N(\T)$ denote
the volume of $[\0, \T]$ and the random number of points in $[\0, \T]$,
respectively, and assume that the $f(y_i)$, $i\in\Z$, are iid with
variance $v$.  Then, asymptotically, $\Var[\hat\mu_f(I, \Phi, \T) |
  \A_n^*] \sim v/N$ and the resulting weights are $w_i(\Phi_i)=
N_i/v$, where $N_i$ denotes the number of points within the $i$-th
realization.
\end{ex}
 Since $N(\T)$
is usually much smaller than $\hat\alpha_{1}(I, \Phi, \T)$, the 
variance $\Var[\mu_f(I, \Phi, \T) | \A_n^*]$ in Proposition 
\ref{prop_toyexample1} is larger than the one in
the hypothetical example in Remark \ref{ex:iid_case}.

In the following example,  we consider arbitrary point locations but 
still assume independence between marks and locations.
\begin{ex}\label{toy_ex_2} 
Let $\tilde\Phi$ be a one-dimensional, 
stationary unmarked point process and $Y$ a
stationary continuous-time process which is independent of
$\tilde\Phi$ and such that $f(Y)$ has finite second moments. We
consider the MPP $\Phi = \{(t, Y(t), 1) : t\in\tilde\Phi \}$. Then
\begin{align*}
\Var&\left[\left.\hat\mu_f(I, \Phi, T) \right| \A_n^* \right] \\
&= 
\frac{\sum_{t_1 \in \Phi_{\g}\cap[0,\,T]} 
   \sum_{s_1 \in \Phi_{\g}\cap[0,\,T]}
    \Cov\bigl[ f(Y(t_1)), f(Y(s_1)) \bigr] 
    n(t_1, \Phi_{\g}, I) n(s_1, \Phi_{\g}, I)
    }{ \left[ \sum_{t_1 \in \Phi_{\g}\cap[0,\,T]} n(t_1, \Phi_{\g}, I) \right]^2 } ,
\end{align*}
where
$n(t_1, \Phi_{\g}, I) = 
\sum_{t_2 \in \Phi_{\g}\backslash\{t_1\}} 
\1_{t_2-t_1 \in I}$.
\end{ex}

\subsection{Remarks}

\begin{rem}
The weighting of multiple realizations and the intrinsically weighted 
means coincide in the following sense:
Let $\Phi_1, \ldots, \Phi_n$ be iid copies of an MPP $\Phi = \{(t_i,
y_i, 1) : i\in\Nb\}$, for which the second mark component equals 1 for
all points. Then the weighting of realizations via $w_i(\Phi_i, \T)$
in the estimator \eqref{est_2} can alternatively be captured by the
second mark component. For $i=1, \ldots, n$, let $\tilde\Phi_i =
\big\{(t, y, w_i^{\rel}(\Phi_i, \T)) : (t,y,1)\in\Phi_i\big\}$, where
$w_i^{\rel}(\Phi_i, \T) = w_i(\Phi_i, \T) / \sum_{k=1}^n w_k(\Phi_k,
\T)$.  Let $\Psi_n$ be the concatenation of the processes
$\tilde\Phi_1, \ldots, \tilde\Phi_n$, each restricted to the
observation window $[\0, \T]$ and concatenated with a buffer of
$\max(I)$ and such that all points of $\Psi_n$ are contained in $[\0,
  \T_n]$ for some $\T_n\in\R^d$.  Then, with $\bw=(w_i(\Phi_i,
\T))_{i=1}^n$, we have $$\hat\mu_f(I, \Psi_n, \T_n) = \hat\mu_f^{n,
  \weighted}(I, \bw, (\Phi_1,\ldots,\Phi_n), \T).$$
\end{rem}

We close this section with a note on the estimation of $\mu_f^{(2)}(r)$
and $\tilde\mu_f^{(2)}(r)$, $r\in\R$. 
\begin{rem}
For most MPPs used in
applications, finding two points of an MPP with a fixed distance $r$
within a bounded observation window, has probability zero.
Then the simplest approach is to apply any of the estimators (\ref{est_0}),
(\ref{est_2}), (\ref{est_1}) or (\ref{est_concat}), with $I$ being a
small interval containing $r$, e.g., $[r-\delta, r+\delta]$ for some
$\delta>0$.  This is equivalent to use (Nadaraya-Watson) kernel
regression with the rectangular kernel, applied to the tuples
$\{(z_1f(y_1),\, \dis(t_2 -t_1)) : (t_1, y_1, z_1), (t_2, y_2, z_2) \in\Phi \}$,
where $\dis(x)=x$ if $x\in\R^1$ and $\dis(x)=\|x\|$ if $x\in\R^d$ with
$d>1$.\\
An obvious generalization is to replace the rectangular kernel by a
general kernel $K_h$ with bandwidth $h$. 
For the basic estimator (\ref{est_0}), this 
yields 
\begin{align*}
  \hat\mu_f(r, \Phi, T)= \frac{\sum^{\neq}_{(t_1, y_1, z_1), (t_2, y_2, z_2) \in\Phi\ t_1\in [0,\, T]} 
    z_1 f(y_1) K_h(r-\dis(t_2-t_1))
  }{\sum^{\neq}_{(t_1, y_1), (t_2, y_2) \in\Phi,\ t_1\in [0,\, T]} K_h(r-\dis(t_2-t_1))
  },
\end{align*}
likewise for the other estimators.  If the support of $K_h$ covers the
whole real line, the denominator is always strictly larger than zero,
which simplifies implementation, but also allows $\hat\mu_f(r, \Phi,
T)$ to be driven by pairs of points whose distance differs largely
from $r$.
\end{rem}

\section{Application to continuous-space processes}
\label{sec:relation_geostatistics}

Picking up the introductory example on continuous-space processes,
taking measurements from such a process with measurement locations
that are possibly irregularly spaced but independent of the underlying
process, leads to a subclass of MPPs.  
At the same time, particularly developed
in the geostatistical context, there exist numerous methods of
inference for continuous-space processes, including methods to account
for biased and preferential sampling.  We compare the concept of
intrinsically weighted means of MPPs to statistical methods for
continuous-space processes in the following.

One of the classical problems in geostatistical applications (e.g.,
\cite{Chiles1999}) is prediction of averages from measurements
$\{(t_i, Y(t_i)): i=1, \ldots, n\}$, where $\{Y_t : t\in T\}$,
$T\subset\R^d$, is a latent second-order stationary random field.
When predicting global moments of $Y$, redundancies in the data can be
excluded via the spatial correlation structure, e.g., the best linear
unbiased estimator (BLUE) for $\Eb Y$ is well-known to be $(\mathbf 1'
\Sigma^{-1} \mathbf 1)^{-1} \cdot \mathbf 1' \Sigma^{-1} \mathbf Y$,
where $\1 = (1, \ldots, 1)'$, $\mathbf Y=(Y(t_1), \ldots, Y(t_n))'$
and $\Sigma=\Cov(Y(t_i), Y(t_j))_{i,j=1}^n$ (e.g.,
\cite[p.179]{Chiles1999}).
More generally, any estimator that
is linear in a transformation $g$ of the data allows
for assigning a different weight to each data point; then the
estimator takes the form $\sum_{i=1}^n z_i g(Y(t_i))$ or
$\sum_{i,j=1}^n z_{ij} g(Y(t_i), Y(t_j))$ (similarly for higher-order
moments).  The weights $z_i$ and $z_{ij}$ are supposed to capture the
spatial or temporal pattern of measurement locations when statistical
inference from irregularly spaced data is carried out.  Similar
weighting procedures are used for declustering and debiasing methods,
cf.\ \cite{Isaaks1989}.
\begin{assert}
Identifying the geostatistical weights $z_i$ with the $z$-component of
the marked point process $\Phi=\{(t_i,y_i,z_i) : i\in\Nb\}$, the
estimator $\sum_{i=1}^n z_i g(Y(t_i))$ of $\Eb g(Y)$ coincides with
the canonical estimator for the weighted mean mark $\mu_f^{(1)}$,
defined by \eqref{mu_firstorder}.
\end{assert}

The geostatistical guiding principle of choosing optimal weights for
aggregation of measurements adheres to the idea that a) there exists
an underlying random field and b) that this field can be measured at
any location without causally influencing the other measurements. It
is important to note that this is far from being satisfied for
processes in which the measurements reflect physical objects that
interact with each other.  Trees in a forest, for example, compete for
resources and if another tree had been added at some point, the
measured characteristics of the surrounding trees would have likely
changed.  Though, with increasing distance, interaction effects
between single objects of an MPP may become negligible and the random
field assumption might be sensible on a larger scale. This perspective
motivates combining classical mean mark estimators for MPPs of the
form $\Phi=\{(t_i,y_i,1) : i\in\Nb\}$ with a geostatistical weighting.
Partitioning the observation window in smaller parts, we assign a
$z$-component to $\Phi$ such that $z_i=z_j$ whenever $t_i$ and $t_j$
belong to the same cell of the partition.  This leads to a classical
unweighted average within each cell and therewith maintains the
information contained in the small-scale pattern of the point
locations.  Between the different cells, we allow for a weighting in
the geostatistical sense and therewith allow to smooth out large-scale
irregularities in the distribution of point locations. We denote the
resulting estimator by $\hat\mu_f^{(1), \geo}$.
\begin{assert}
Considering a realization of $\Phi$ as a collection of realizations of
a possibly non-ergodic MPP on smaller observation windows
corresponding to the above partition, the form of $\hat\mu_f^{(1),
  \geo}$ coincides with that of $\hat\mu_f^n$ and $\hat\mu_f^{n,
  \weighted}$, which estimate the average mean mark 
$\tilde\mu_f^{(1)}$ (see Definition \ref{def:mu_non-erg}) instead of 
the classical mean mark $\mu_f^{(1)}$.
\end{assert}
The application of such a weighting scheme is particularly of interest 
when  the underlying process jumps between different
regimes that differ substantially from each other, e.g., w.r.t.\ the
intensity of point locations. 
In summary, applying the geostatistical idea of declustering in the
MPP context in a sense corresponds to the concept 
of non-ergodic modeling.

\medskip

To avoid possible confusion, we conclude this section with a final
remark. 
\begin{rem}
For certain choices of $f$, the random field counterpart of
$\mu_f^{(2)}$ is well-defined. For $f(y_1, y_2) =y_1y_2$, for
instance, the counterpart is the ordinary (non-centered) 
covariance function.  If $f$
only depends on one of the two marks of a pair of points,
$\mu_f^{(2)}$ implicitly conditions on the existence of other points
and there is no sensible way of interpreting a suchlike statistic in a
random field context, where there exist values at all points of the
index space. Nevertheless, the geostatistical idea of
variance-minimizing weights can be applied to $\mu_f^{(2)}$ by a
simple mean squared error approach.
\end{rem}

\section{Discussion}\label{sec:discussion}
The MPP summary statistics considered in this paper are (weighted)
mean marks. In practice, the choice of weights is not always clear, for 
example when data from different stochastic sources are combined. In
Section \ref{sec:relation_geostatistics}, we point out that, if there
was an underlying continuous-time process from which the data were
generated by a random sampling procedure, then the \emph{mean of
interest} would rather be the temporal average over the whole index
space instead of the average over all sampling locations. The weights 
might then be chosen to compensate for the irregular distribution of point
locations. Though, the assumption
of a continuous-time background process is problematic if the points
represent physical objects that influence each other. Then, the 
mean of interest might include the randomness of the point pattern, 
as it is reflected by the MPP moment measures $\alpha_f^{(2)}$.\\
Related questions arises when multiple realizations of a non-ergodic
MPP are considered: Should the definition of mean include possibly
different intensities of points between different ergodicity classes
or not?  A non-ergodic MPP can be seen as a hierarchical model and
expectation functionals w.r.t.\ the point process can naturally be
replaced by two-step expectations by averaging within each ergodicity
class first and then aggregating the different classes (cf.\
Section \ref{sec:nonergodic_def}). This alternative definition
filters out the
differences w.r.t.\ the point location patterns between different ergodicity
classes. Which definition of mean should be chosen eventually
depends on the purpose of the characteristic at hand and on the
intended interpretation.



\appendix
\section{Ergodic theory}\label{ergodic_theory}

Ergodicity is a mixing property that can be defined in the very
general context of dynamical systems. A MPP on $\R^d$ together with the
group of $\R^d$-indexed shift operators is a special case of a
dynamical system.

We denote by $\M_0$ the set of all locally finite counting measures on
$\R^d \times \R$, and by $\MM_0$ the smallest $\sigma$-algebra on
$\M_0$ that makes all mappings $\M_0\rightarrow\Nb_0\cup \infty$,
$\varphi\mapsto\varphi(S)$, measurable.  Formally, a MPP $\Phi$ is a
measurable mapping from some probability space $(\Omega, \A, P)$ into
$(\M_0, \MM_0)$ and we can identify $(\Omega, \A)$ with $(\M_0,
\MM_0)$ in the usual way. Let $\T = \{T_x : x\in\R^d \}$ with 
\begin{align}
(T_x \varphi)(B\times L) = \varphi((B+x), L), \qquad B\in \BB^d, L\in\R.
\label{MPPshift}
\end{align}
Recall that $\Phi$ is said to be stationary if the induced probability
measure $P^\Phi$ is $\T$-invariant.  Further, a stationary MPP $\Phi$
is called ergodic if $P^\Phi(A)$ is either zero or one for all
$\T$-invariant sets $A\in\MM_0$. Let $\A_0\subset\MM_0$ be the
sub-$\sigma$-algebra of all $\T$-invariant sets in
$\MM_0$, i.e., $A=T^{-1}A$ for all $A\in\A_0$ and $T\in\T$.

The following theorem is commonly termed \emph{pointwise} or \emph{individual
ergodic theorem} in literature and establishes almost sure
convergence of a certain average of values of a random variable $X$.

\begin{defn}[\protect{Def.~12.2.I in \cite{Daley2008}}] \label{av_seq}
  An increasing sequence 
  of bounded convex Borel sets
  $W_n\subset\R^d$ is called \emph{convex averaging sequence in
    $\R^d$} if the maximal radius of a ball contained in $W_n$ goes
  to infinity if $n$ increases.
\end{defn}

\begin{thm}[\protect{Prop.~12.2.II \cite{Daley2008}}]\label{ind_erg}
  Let $(\Omega, \A, P)$ be a probability space and $\T = \{T_x : x\in
  \R^d\}$ a group of measure-preserving transformations
  acting on $(\Omega, \A, P)$ such that the mapping $(T_x, \omega)
  \mapsto T_x\omega$ is jointly measurable, i.e., $(\BB(\T)\otimes\A,
  \A)$-measurable. (Multiplication in $\T$ is given by $T_xT_y =
  T_{x+y}$.)  Let $\{W_n\}_{n\in\Nb}$ be a convex averaging sequence in $\R^d$
  and $\A_0$ the $\sigma$-algebra of $\T$-invariant events. Then for all
  real-valued integrable functions $X$ on $(\Omega, \A, P)$
  \begin{align*}
    \bar X_n = \frac{1}{\nu(W_n)} \int_{W_n} X(T_x\omega) \,\nu(\dd
    x) \stackrel{\text{a.s.}}{\longrightarrow} \Eb(X\,|\,
    \A_0),\qquad n\rightarrow \infty. 
  \end{align*} 
  If $X$ is additionally $L_p$-integrable, then $\Eb(X\,|\,
    \A_0)$ is also the $L_p$-limit of $\bar X_n$.
\end{thm}

\begin{rem}
  If $P$ is ergodic (i.e., $P(A)\in\{0, 1\} \ \forall A\in\A_0$)
  then $\Eb(X\,|\, \A_0)$ reduces to the constant $\Eb X$. Loosely
  speaking, this means that a suitable average over transformations
  of a single realization converges to the expectation over the state
  space~$\Omega$.
\end{rem}

While Theorem \ref{ind_erg} refers to a general probability space with
a general group of transformations action on it, the following
Proposition relates this results to the context of MPPs on $\R^d$, in
which the transformations $T_x$, $x\in\R^d$, are given by shifts of
the whole point pattern by the vector $x$. Here, the point is that the
index $x\in\R^d$ has a direct geometric meaning when $T_x$ is applied
to a realization $\varphi$ of $\Phi$. This yields convergence of
spatial averages within a single realization of the MPP to the state space
mean.

The proof of the following Proposition is based on a simple sandwich
argument, which can also be used for other consistency statements. We
include the proof here, because to our knowledge, it is not available
in this form in pertinent literature. A similar assertion can be found 
in \cite[Thm.~12.2.IV]{Daley2008}.

\begin{prop}\label{cor_averaging2}
Let $\Phi$ be stationary and ergodic and $\T$ as in Theorem
\ref{ind_erg}.  Let $f:\R^d \times \R \times \M_0\rightarrow\R$ be a
non-negative function that satisfies $f(t-x, y, T_x\varphi) = f(t, y,
\varphi)$ for all $t, x\in\R^d, y\in\R$, and that is integrable
w.r.t.\ to the marked Campbell measure $C(B\times L \times M) =
\Eb\bigl[\Phi((B\cap[0, 1]^d)\times L)\1_M(\Phi)\bigr]$, $B\in\BB^d$,
$L\in\LL$, $M\in\MM_0$. We define random variables $X, X_n : \M_0
\rightarrow \R$ by
\begin{align*}
  X(\varphi) &=   \sum_{(t, y) \in\varphi, \ t\in [0,\, 1]^d} f(t, y, \varphi) \\
  X_n(\varphi) &=  \frac{1}{n^d} \sum_{(t, y) \in\varphi, \ t\in [0, n]^d} 
  f(t, y, \varphi).
\end{align*}
Then $X_n$ converges to $\Eb^\Phi X$ almost surely if $n \rightarrow \infty$.
\begin{proof}
An extension of the classical Campbell theorem 
(e.g., Lem.~13.1.II in \cite{Daley2008}) guarantees that $\Eb|X|<\infty$ if
$f$ is integrable w.r.t.\ the Campbell measure.  The $W_n=[0, n]^d$
obviously form an averaging sequence and
\begin{align}
  X_n(\varphi) 
  &=  \frac{1}{\nu(W_n)} \sum_{(t, y) \in\varphi, \ t\in W_n} f(t, y, \varphi)
  \int_{\R^d} \1_{[t,\, t+1]}(x)\, \nu(\dd x) \notag\\
  &=  \frac{1}{\nu(W_n)} \int_{\R^d} 
   \sum_{(t, y) \in\varphi, \ t\in W_n \cap [x-1, \, x]}
   f(t, y, \varphi) \, \nu(\dd x), \label{xn}
\end{align}
where $x\pm 1$ for $x\in\R^d$ is defined component-wise.  Note that
the integrand on the RHS equals 0 whenever $W_n \cap [x-1,
  \, x] = \varnothing$, which means that $x$ is not contained in $W_n
\oplus [0, 1]^d$, which is, on its part, a subset of $W_{n+1}$. Thus,
we can shrink the region of integration to $W_{n+1}$ without changing
the integral. If we then drop the condition `$t\in W_n$' under the
summation sign, we enlarge the whole expression since $f$ is
non-negative, i.e.\
\begin{align}
  X_n(\varphi) &\leq    
  \frac{1}{\nu(W_n)} \int_{W_{n+1}} \sum_{(t, y) \in\varphi, \ t\in [x-1, \, x]}
  f(t, y, \varphi) \, \nu(\dd x) \notag\\
  &= \frac{1}{\nu(W_n)} \int_{W_{n+1}} \sum_{(t, y) \in T_{x-1}\varphi, \ t\in [0, \, 1]^d}
  f(t, y, T_{x-1}\varphi) \, \nu(\dd x)  \notag\\
  &= \frac{\nu(W_{n+1})}{\nu(W_n)} \frac{1}{\nu(W_{n+1})} 
  \int_{W_{n+1}-1} X(T_x\varphi) \, \nu(\dd x),  \label{upper_b}
\end{align}
where the second equation uses that $f(t-x, y, T_x\varphi) = f(t, y,
\varphi)$ and the last equation uses that $\nu$ is shift-invariant.
Since the ratio $\nu(W_{n+1})/\nu(W_n)$ converges to 1, Theorem
\ref{ind_erg} yields that the RHS of (\ref{upper_b})
converges to $\Eb^\Phi(X\,|\, \A_0)$ for almost all
$\varphi\in\M_0$. Since $\Phi$ was assumed to be ergodic, this
conditional expectation equals $\Eb^\Phi X$.\\ Similarly, if we
restrict integration in (\ref{xn}) to the set $W_{n-1}$, we reduce the
value of the integral. Since $W_{n-1} \oplus [-1, 0]^d \subset W_n$, we
can again drop the condition `$t\in W_n$' under the summation sign and
by the same argument as before, we have
\begin{align*}
  X_n(\varphi) &\geq    
  \frac{1}{\nu(W_n)} \int_{W_{n-1}} \sum_{(t, y) \in\varphi, \ t\in [x, \, x+1]}
  f(t, y, \varphi) \, \nu(\dd x) \
   \stackrel{n\rightarrow\infty}{\longrightarrow}  \
  \Eb^\Phi X 
\end{align*}
for almost all $\varphi\in\M_0$.  Thus, we have a sandwich relation
for $X_n(\varphi)$ and can conclude that $X_n \rightarrow \Eb^\Phi X$
a.s.
\end{proof}
\end{prop}
Note that the convex averaging sequence $\{[0, n]^d\}_{n\in\Nb}$ in
Proposition \ref{cor_averaging2} can be replaced by any sequence
$\{W \oplus nV\}_{n\in\Nb}$ with  
$W$ a bounded Borel set and $V\subset\R^d$ a convex and
bounded set with $\nu(V)>0$ and $0\in V$.
\bigskip

In case that $\Phi$ is not ergodic, the following results provide a
representation of $\Phi$ as a mixture of a set of ergodic MPPs.
To this end, let $\cP$ ($\cP_{\erg}$ resp.) denote the set of all
probability measures on $(\M_0, \MM_0)$ induced by stationary (and
ergodic) MPPs and let $\Pi_{\erg}$ be the smallest $\sigma$-algebra
making all mappings $\cP_{\erg}\rightarrow [0, 1]$, $P\mapsto P(A)$,
measurable.  We say that $\T$ fulfills the condition (LocCompGrp) if
$\T$ is a locally compact, second-countable Hausdorff group of jointly
measurable, surjective transformations.

From \cite{Farrell1962} 
we can extract the very general result

\begin{thm}\label{decomp_theorem}
  Let $(\Omega, \A)$ be a measurable space with $\Omega$ a complete
  separable metric space and $\A$ its Borel-$\sigma$-algebra. Let $\T$
  be a set of measurable transformations of $\Omega$ satisfying the
  condition (LocCompGrp) and let $P\in\cP$.  Here, $\cP$ ($\cP_{\erg}$
  resp.) is the set of all $\T$-invariant (and ergodic) probability
  measures on $(\Omega, \A)$. Then there is a unique probability
  measure $\lambda_P$ on $(\cP_{\erg}, \Pi_{\erg})$ and a
  $\cP_{\erg}$-valued random variable $Q_P$ s.t.
  $$P(A) = \int_{\cP_{\erg}} Q(A) \, \lambda_P(\dd Q)
  = \int_\Omega Q_P(\omega)(A) \, P(\dd \omega) \qquad \forall
  A\in\A,$$
  i.e.,  $\lambda_P$ is the distribution of $Q_P$. 
\end{thm}

In the context of MPPs on $\R^d$, the group $\T$ of shifts, as defined
in (\ref{MPPshift}), obviously fulfills the condition (LocCompGrp),
and since $\M_0$ is a complete separable metric space and $\MM_0$ its
Borel-$\sigma$-algebra (e.g., \cite{Kallenberg1983}), Theorem
\ref{decomp_theorem} can directly be applied, which yields a
decomposition of the non-ergodic MPP $\Phi\sim P$:
$$P(M) = \int_{\cP_{\erg}} Q(M) \, \lambda(\dd Q) \qquad
\forall M\in\MM_0.$$  
Note that each $Q$ induces a new ergodic MPP $\Phi_Q :
\Omega\rightarrow\M_0$ which is given implicitly by $P(\Phi_Q \in M) =
Q(M)$, $M\in\MM_0$.
By the second representation in Theorem~\ref{decomp_theorem}, we can
also consider $Q$ as a random variable on $(\M_0, \MM_0, P)$
with distribution $\lambda = \lambda_{P}$. Thus, $\Phi$ and
$Q^\Phi$ have a joint distribution and the conditional distribution of
$\Phi$ given $Q$ is well-defined:
\begin{align*}
P(\cdot \,|\, Q = Q^*) = Q^*(\cdot). 
\end{align*}

\section{Proof of Theorem \ref{CLT_MPP}}\label{app_proof_CLT}

The following lemma generalizes the classical individual ergodic
theorem \cite[Prop.\ 12.2.II]{Daley2008} to a situation in which the
thinning of the point process depends on the size of the observation
window.
\begin{lem}
\label{ergodic_theorem}
Let $\Phi$ be a stationary and ergodic MPP on $\R$ with real-valued
marks and let $(u_T)_{T\geq 0}$ be a family of non-negative
non-decreasing numbers such that 
\begin{align}
  \frac{T^{-1}\hat\alpha_{1}(I, \Phi, T) -\lambda}{
  \Eb_{\Phi}\hat\alpha_{f_{\condit, {u_T}}}(I, \Phi, 1)} 
\to 0 \quad \text{a.s.} \ (T\to\infty).
\label{conv_speed}
\end{align}
Then, for $T\to\infty$, we have the
almost sure convergence
\begin{align*}
\frac{\hat\alpha_{f_{\condit, {u_T}}}(I, \Phi, T)}{
  T \Eb_{\Phi}\hat\alpha_{f_{\condit, {u_T}}}(I, \Phi, 1)} 
\longrightarrow 1.
\end{align*}
\end{lem}
Note that the almost sure convergence 
$(\lambda T)^{-1}\hat\alpha_{1}(I, \Phi,
T)\to 1$ as $T\to\infty$ follows from the classical individual
ergodic theorem (e.g., \cite[Prop.\ 12.2.II]{Daley2008}).
\begin{proof}[Proof of Lemma \ref{ergodic_theorem}]
 With $g_u(y)
= 1 - f_{\condit, u}(y)$, $y\in\R$, we obtain the almost sure
convergence 
\begin{align*}
\frac{\hat\alpha_{g_{u_T}}(I, \Phi, T)}{T \Eb_\Phi
\hat\alpha_{g_{u_T}}(I, \Phi, 1)} \to 1
\end{align*}
from \cite[Prop.\ 12.2.VII]{Daley2008} and the subsequent remarks.
Further, $\lambda  =  \Eb_\Phi \hat\alpha_{1}(I, \Phi, 1) = 
\Eb_\Phi \hat\alpha_{f_{\condit, {u_T}}}(I, \Phi, 1) + \Eb_\Phi
\hat\alpha_{g_{u_T}}(I, \Phi, 1)$.  Hence,
\begin{align*}
\frac{\hat\alpha_{f_{\condit, {u_T}}}(I, \Phi, T)}{
  T \Eb_{\Phi}\hat\alpha_{f_{\condit, {u_T}}}(I, \Phi, 1)}
&= \frac{\hat\alpha_{1}(I, \Phi, T) - \hat\alpha_{g_{u_T}}(I, \Phi, T)}{
T \Eb_{\Phi}\hat\alpha_{f_{\condit, {u_T}}}(I, \Phi, 1)} \\
&= \frac{  \lambda \frac{\hat\alpha_{1}(I, \Phi, T)}{\lambda T} - 
   \Eb_\Phi\hat\alpha_{g_{u_T}}(I, \Phi, 1)\frac{\hat\alpha_{g_{u_T}}(I, \Phi, T)}{
     T \Eb_\Phi\hat\alpha_{g_{u_T}}(I, \Phi, 1)}}{
 \Eb_{\Phi}\hat\alpha_{f_{\condit, {u_T}}}(I, \Phi, 1)} 
\end{align*}
and the RHS converges to 1 as long as $\Eb_{\Phi}\hat\alpha_{1,
  f_{\condit, {u_T}}}(I, \Phi, 1)$ converges to 0 at a slower rate
(in the sense of \eqref{conv_speed}) than $\frac{\hat\alpha_{1}(I,
  \Phi, T)}{\lambda T}$ and $\frac{\hat\alpha_{g_{u_T}}(I, \Phi,
  T)}{T \Eb_\Phi\hat\alpha_{g_{u_T}}(I, \Phi, 1)}$ approach 1.
\end{proof}

\begin{proof}[Proof of Theorem\ \protect{\ref{CLT_MPP}}]
We have
\begin{align*}
\frac{\hat\alpha_{f_{u_T}}^*(I, \Phi, T)}{\sqrt{\hat\alpha_{f_{\condit,u_T}}(I, \Phi, T)}}
= \frac{\hat\alpha_{f_{u_T}}^*(I, \Phi, T)}{\sqrt{[\lambda_{u_T} T]}} 
  \frac{\sqrt{[\lambda_{u_T} T]}}{\sqrt{\lambda_{u_T} T}} 
  \frac{\sqrt{\lambda_{u_T} T}}{\sqrt{\hat\alpha_{f_{\condit,u_T}}(I, \Phi, T)}}
\end{align*}
and by Lemma \ref{ergodic_theorem}, the last factor converges to 1. (Here, for $a\geq 0$, $[a]$ denotes
the smallest integer $\geq a$.) Hence, for convergence of
the LHS it is sufficient to show that  
$\hat\alpha_{f_{u_T}}^*(I, \Phi, T)/\sqrt{[\lambda_{u_T} T]}$ 
converges to a Gaussian
variable.  
According to \cite[Lemma~2.1, Lemma~2.3]{Kallenberg1983}, we can write
$\Phi$ as a sum of Dirac measures $\delta_{(T_i, Y_i)}$, $i\in\Nb$,
with random vectors $(T_i, Y_i)$ and $T_1\leq T_2 \leq \ldots$ If only
a finite observation window $[0, T]$ is considered, the number of
summands $N(T)$ is also finite but random.
Then we introduce a modified version of
$\hat\alpha_{f_u}^*(I, \Phi, T)$, in which the sum is cut
after a fixed number $N_{\max}\in\Nb$ of terms:
\begin{align*}
  \hat\alpha_{f_u}^{*, N_{\max}}(I, \Phi, T) 
  &=\sum_{i=1}^{N(T)} \sum_{j=1}^{N(T)}  
    \Bigl( f_u(Y_i) - \mu_{f_u, f_{\condit, u}}^{(2)}(I) \Bigr)
    \cdot f_{\condit, u}(Y_i)\cdot\1_{T_j - T_i \in I}\\
   &\qquad  \cdot\1_{\left[\sum_{i'=1}^{i-1} \sum_{j'=1}^{N(T)}
       f_{\condit, u}(Y_i)\1_{T_{j'} - T_{i'} \in I} +    
       \sum_{j'=1}^{j}
       f_{\condit, u}(Y_i)\1_{T_{j'} - T_{i} \in I} \leq N_{\max}\right]}.
\end{align*}
Then we have
\begin{align}
  \frac{\hat\alpha_{f_{u_T}}^*(I, \Phi, T)}{\sqrt{[\lambda_{u_T} T]}} 
  &= \frac{\hat\alpha_{f_{u_T}}^{*, [\lambda_{u_T} T]}(I, \Phi, \infty)}{\sqrt{[\lambda_{u_T} T]}} 
+ \frac{\hat\alpha_{f_{u_T}}^*(I, \Phi, T) - 
    \hat\alpha_{f_{u_T}}^{*, [\lambda_{u_T} T]}(I, \Phi, \infty)
  }{\sqrt{[\lambda_{u_T} T]}}  \label{CLT_MPP_proof_1} 
\end{align}
and the first summand of the RHS contains a non-random number of
summands (namely $[\lambda_{u_T} T]$).  By the minimum distance
assumption in condition ($m$-dependent Random Field Model), each mark
$Y_i$ occurs at most $|I|/d_0$ times in  
$\hat\alpha_{f_{u_T}}^{*, [\lambda_{u_T} T]}(I, \Phi, \infty)$. By the
finite-range assumption on the covariance function of
the underlying random field, the sequence $(Y_i)_{i\in \Nb}$ is
$[h_0/d_0]$-dependent.  Hence, the sequence of summands in
$\hat\alpha_{f_{u_T}}^{*, [\lambda_{u_T} T]}(\Phi,
I, \infty)$ is $[|I|h_0 / d_0^2]$-dependent.  
By assumption, the first four  moments of the excesses 
$Z_i = [f_{u_T}(Y_i) \,|\, f(Y_i) > u_T]$ exist and converge to some constant in
$(0, \infty)$ as $T\to\infty$.  Then the sequence of summands in
$\hat\alpha_{f_{u_T}}^{*, [\lambda_{u_T} T]}(\Phi,
I, \infty)$ satisfies the assumptions of Berk's
 CLT for triangular arrays of $m$-dependent
random variables \cite{Berk1973} and thus, for $T\to\infty$, 
$\hat\alpha_{f_{u_T}}^{*, [\lambda_{u_T} T]}(I, \Phi,
\infty)/\sqrt{[\lambda_{u_T} T]}$ approaches a Gaussian distribution
with zero mean and variance 
\begin{align*}
  u_\infty = \lim_{T\to\infty} \Var \left[ 
    \hat\alpha_{f_{u_T}}^{*, [\lambda_{u_T} T]}(I, \Phi, \infty)\right]  \big/ ([\lambda_{u_T} T]).
\end{align*}

Next, we show that the second summand in (\ref{CLT_MPP_proof_1}) 
converges to 0 in probability.  We use the notation 
$\Delta \alpha_{f_u} =
\hat\alpha_{f_u}^{*}(I, \Phi, T) - \hat\alpha_{f_u}^{*, [\lambda_{u_T} T]}(I, \Phi, \infty)$
and 
$\Delta \alpha_{1} =
\hat\alpha_{f_{\condit, u}}(I, \Phi, T) - \hat\alpha_{f_{\condit, u}}^{[\lambda_{u_T} T]}(I, \Phi, \infty)$
and consider
\begin{align}
  &\P( |\Delta \alpha_{f_{u_T}}|  \geq \varepsilon \sqrt{[\lambda_{u_T} T]} )  \notag\\
  &= \P\Bigl(|\Delta \alpha_{f_{u_T}}|  \geq \varepsilon \sqrt{[\lambda_{u_T} T]} \,\Big|\, 
  |\Delta \alpha_{1}| \geq \varepsilon [\lambda_{u_T} T] \Bigr)
  \cdot \P\bigl(|\Delta \alpha_{1}| \geq \varepsilon [\lambda_{u_T} T] \bigr) \notag\\
  &+ \P\Bigl(|\Delta \alpha_{f_{u_T}}|  \geq \varepsilon \sqrt{[\lambda_{u_T} T]} \,\Big|\, 
  |\Delta \alpha_{1}| < \varepsilon [\lambda_{u_T} T] \Bigr)
  \cdot \P\bigl(|\Delta \alpha_{1}| < \varepsilon [\lambda_{u_T} T] \bigr)\notag\\
  &\leq \P\bigl(|\Delta \alpha_{1}| \geq \varepsilon [\lambda_{u_T} T] \bigr)
  + \P\Bigl(|\Delta \alpha_{f_{u_T}}|  \geq \varepsilon \sqrt{[\lambda_{u_T} T]} \,\Big|\, 
  |\Delta \alpha_{1}| < \varepsilon [\lambda_{u_T} T] \Bigr) \label{CLT_MPP_proof_2}
\end{align}
Note that $\hat\alpha_{f_{\condit, {u_T}}}^{ [\lambda_{u_T} T]}(I, \Phi, \infty)= [\lambda_{u_T} T]$ and hence
\begin{align}
  \P\bigl(|\Delta  \alpha_{1}| \geq \varepsilon [\lambda_{u_T} T] \bigr)
  = \P\left(\left|\hat\alpha_{f_{\condit, {u_T}}}(I, \Phi, T)\big/[\lambda_{u_T} T] -1\right| 
  \geq \varepsilon \right)
  \rightarrow 0 \quad \text{for } T\rightarrow \infty.
  \label{conv_delta_s0}
\end{align}
To estimate the the last summand in (\ref{CLT_MPP_proof_2}), we use
again that the sequence $(Y_i)_{i\in\Nb}$ is $[h_0/d_0]$-dependent and
that the number of points in any interval of length $|I|$ is bounded
by $c=|I|/d_0$. This means that each term $f_{u_T}(Y_i)$ occurs at most 
$c$ times in the sum $\Delta
\alpha_{f_{u_T}}$.  Obviously, the variance of $\Delta
\alpha_{f_{u_T}}$, or more generally all even
centered moments of $\Delta \alpha_{f_{u_T}}$,
become maximal, if this boundary is bailed, i.e., if for a given total
number $\Delta \alpha_{1}$ of summands, only
$\left[\Delta \alpha_{1}/c\right]$
different $Y_i$ are involved.  With $Z_i^* = Z_i - \Eb Z_i = [f_{u_T}(Y_i) \,\big|\, 
f(Y_i) > {u_T}] - e({u_T})$, where $e(u) = \Eb
\left[f_{u}(Y(0)) \,\big|\, f(Y(0)) > u\right]$, we get
\begin{align*}
  &\P\bigl(|\Delta \alpha_{f_{u_T}}|  \geq \varepsilon \sqrt{[\lambda_{u_T} T]} \,\big|\, 
  |\Delta \alpha_{1}| < \varepsilon [\lambda_{u_T} T] \bigr) \\
  &=\P\left(|\Delta \alpha_{f_{u_T}}|^4  \geq \varepsilon^4 [\lambda_{u_T} T]^2 \,\Big|\, 
  |\Delta \alpha_{1}| < \varepsilon [\lambda_{u_T} T] \right) 
  \\\displaybreak[0]  &\leq \textstyle
  \P\left(\left|\sum_{i=1}^{[\varepsilon [\lambda_{u_T} T] c^{-1}]} cZ_i^*\right|^4 
  \geq \varepsilon^4 [\lambda_{u_T} T]^2 \right)\\ 
  &\leq  c^4 
  \sum_{i,j,k,l=1}^{[\varepsilon [\lambda_{u_T} T] c^{-1}]}
  \Eb(Z_i^* Z_j^* Z_k^*Z_l^*) 
  \cdot (\varepsilon^4[\lambda_{u_T} T]^2)^{-1}\\
  &\leq  c^4 \cdot
  [\varepsilon [\lambda_{u_T} T] c^{-1}] 
  \cdot \left(\frac{h_0}{d_0}\right)^3 \Eb\left[(Z_1^*)^4\right]
  \cdot (\varepsilon^4[\lambda_{u_T} T]^2)^{-1}\\
  &= (\lambda_{u_T} T)^{-1} \varepsilon^{-3} \left(c\frac{h_0}{d_0}\right)^3 
  \Eb\left[(Z_1^*)^4\right] (1+o(1)) \longrightarrow 0, \quad (T\to\infty).
\end{align*}
Plugging this and (\ref{conv_delta_s0}) into (\ref{CLT_MPP_proof_2})
yields that $\Delta \alpha_{f_{u_T}}/\sqrt{[\lambda_{u_T} T]} \rightarrow 0$ in probability.
\end{proof}

\section{Proofs of Examples in Section \ref{sec:estimators}}\label{app_proof_ex}

\begin{proof}[Proof of Example \ref{prop_toyexample1}]
For $|I|$ and $\T$ large, we have $\hat\alpha_{1}(I, \Phi, \T) \sim
N\cdot N |I|/v_\T$ and each distinct summand in $\hat\alpha_{f}(I, \Phi, \T)$ 
occurs $N |I| / v_\T \sim \hat\alpha_{1}(I, \Phi, \T)/N$ times. Thus,
$\hat\alpha_{f}(I, \Phi, \T)
  \sim \hat\alpha_{1}(I, \Phi, \T) \sum_{i=1}^N f(y_i)/N$ and 
$\Var[\hat\alpha_{f}(I, \Phi, \T) | \A_n^*]\sim \hat\alpha_{1}(I, \Phi, \T)^2 v / N$ 
\end{proof}

\begin{proof}[Proof of Example \ref{toy_ex_2}]
We have
\begin{align*}
\Eb &\bigl[ \hat\alpha_{f}(I, \Phi, T)/ \hat\alpha_{1}(I, \Phi, T) \,\big|\, \A_n^* \bigr]\notag \\[.5em]
&= \hat\alpha_{1}(I, \Phi, T)^{-1} \cdot \textstyle
   \Eb \left[\left.\sum_{(t_1, y_1, z_1), (t_2, y_2, z_2) \in\Phi,\ t_1\in[0,\,T]} 
    z_1 f(y_1)\cdot\1_{t_2 - t_1 \in I} 
    \right| 
    \A_n^* \right]\notag \\[.5em]
&= \hat\alpha_{1}(I, \Phi, T)^{-1} \cdot \textstyle 
   \sum_{t_1 \in \Phi_{\g}\cap[0,\,T]}   
   \cdot \#\{t_2\in\Phi_{\g} : t_2-t_1 \in I\} \cdot 
   \Eb \left[ f(Y(t_1))  | \A_n^* \right] \notag\\[.5em]
&= \Eb f(Y(0)). 
\end{align*}
and 
\begin{align*}
  &\Eb \bigl[ \hat\alpha_{f}(I, \Phi, T)^2 \,|\, \A_n^* \bigr] \notag\\
  &=\Eb \Big[ 
    \sum_{t_1, s_1 \in \Phi_{\g}\cap[0,\,T]} 
    f(Y(t_1) f(Y(s_1)) \notag\\[-1em]
    & \hspace*{10em}
    \cdot \#\{t_2\in\Phi_{\g} : t_2-t_1 \in I\} 
    \cdot \#\{s_2\in\Phi_{\g} : s_2-s_1 \in I\} \,\Big|\, \A_n^* \Big] 
  \notag\\
&=  \sum_{t_1, s_1 \in \Phi_{\g}\cap[0,\,T]}  
    n(t_1, \Phi_{\g}, I) n(s_1, \Phi_{\g}, I)
    \notag\\[-1em]
    & \hspace*{10em}
    \cdot \Eb \bigl[ f(Y(t_1) f(Y(s_1)) \,|\, \A_n^* \bigr] 
    \displaybreak[0]\notag\\ 
&=  \sum_{t_1, s_1 \in \Phi_{\g}\cap[0,\,T]}  
    n(t_1, \Phi_{\g}, I) n(s_1, \Phi_{\g}, I)
    \notag\\[-1em]
    & \hspace*{10em}
    \cdot \Bigl[ \Eb \left[ f(Y(0)) | \A_n^* \right]^2 
      + \Cov \bigl[ f(Y(t_1), f(Y(s_1)) \,|\, \A_n^* \bigr] \Bigr]
    \notag\\
&=  \sum_{t_1, s_1 \in \Phi_{\g}\cap[0,\,T]}  
    n(t_1, \Phi_{\g}, I) n(s_1, \Phi_{\g}, I)
    \cdot \Cov \bigl[ f(Y(t_1), f(Y(s_1)) \bigr] \notag\\
    & \qquad + (\Eb f(Y(0)))^2 \hat\alpha_{1}(I, \Phi, T)^{2} .
\end{align*}
Hence,
\begin{align*}
  \Var \bigl[ &\hat\alpha_{f}(I, \Phi, T)/\hat\alpha_{1}(I, \Phi, T) \,|\, \A_n^* \bigr]\\  
  &=  \Eb \bigl[ (\hat\alpha_{f}(I, \Phi, T)/\hat\alpha_{1}(I, \Phi, T)^2 \,|\, \A_n^* \bigr]
  - \bigl(\Eb [\hat\alpha_{f}(I, \Phi, T)/\hat\alpha_{1}(I, \Phi, T) \,|\, \A_n^*]\bigr)^2   
  \notag\\
  &= \hat\alpha_{1}(I, \Phi, T)^{-2} \cdot \Eb \bigl[ \hat\alpha_{f}(I, \Phi, T)^2 \,|\, 
    \A_n^* \bigr] - 
  (\Eb f(Y(0)))^2\notag\\
  & = \hat\alpha_{1}(I, \Phi, T)^{-2}
    \sum_{t_1, s_1 \in \Phi_{\g}\cap[0,\,T]}  
    n(t_1, \Phi_{\g}, I) n(s_1, \Phi_{\g}, I)
    \cdot \Cov \bigl[ f(Y(t_1), f(Y(s_1)) \bigr].
\end{align*}
\end{proof}

\acks
The authors are indebted to Katrin Meyer for pointing their attention 
to the example in \cite{Begon1990}. A. Malinowski has been financially
supported the German Science Foundation (DFG), Research Training Group
1644 `Scaling problems in Statistics'. 

\bibliographystyle{apt} 
\bibliography{bib,../diplomarbeit/mybibtex}

\end{document}